\documentclass{article}

\UseRawInputEncoding
\usepackage{CJK}
\usepackage[utf8]{inputenc}
\usepackage{authblk}
\usepackage{amsmath, amssymb, amsthm, enumerate}
\usepackage{algorithm}
\usepackage{algpseudocode}
\usepackage{multirow}
\usepackage{array}
\usepackage{graphicx} %use graph format
\usepackage{epstopdf}
\usepackage{float}

\usepackage{amsmath,amssymb,amsthm,enumerate,mathrsfs}
\newtheorem{theorem}{Theorem}[section]
\newtheorem{corollary}[theorem]{Corollary}
\newtheorem{lemma}[theorem]{Lemma}

\theoremstyle{definition}

\newtheorem{assum}{Assumption}

\newtheorem{remark}[theorem]{Remark}

\numberwithin{equation}{section}

\begin{document}
\makeatletter

\title{An Improved Spectral Conjugate Gradient Algorithm Based on A Modified Wolfe Line Search}
\begin{CJK}{UTF8}{gbsn}
	\author[a]{Hao Wu \thanks{E-mail address: wuhoo104@nuaa.edu.cn}
	}
	
	\author[a]{Liping Wang\thanks{E-mail address: wlpmath@nuaa.edu.cn}}
	
	\author[b]{Hongchao Zhang\thanks{E-mail address: hozhang@math.lsu.edu}}
	
	\affil[a]{School of Mathematics, Nanjing University of Aeronautics and Astronautics, Nanjing, China.}
	
	\affil[b]{Department of Mathematics, Louisiana State University Baton Rouge, LA, USA.}
	\maketitle

\end{CJK}

\vspace{2mm}

\footnotesize{
\noindent\begin{minipage}{14cm}
{\bf Abstract:}
In this paper, we combine the $m$th-order Taylor expansion of the objective function with cubic Hermite interpolation conditions. Then, we derive a series of modified secant equations with higher accuracy in approximation of the Hessian matrix of the objective function. A modified Wolfe line search is also developed. It overcomes the weakness of the typical constraint which is imposed on modified secant equations and to keep the curvature condition met. Therefore, based on the modified secant equation and Wolfe line search, an improved spectral conjugate gradient algorithm is proposed. Under some mild assumptions, the algorithm is showed to be globally convergent for general nonconvex functions. Numerical results are also reported for verifying the effectiveness.
\end{minipage}
 \\[5mm]


\hbox to14cm{\hrulefill}\par

%%%%%%%%%%%%%%%%%%%%%%%%%%%%%%%%%%%%%%%%%%%%%%%%%%%%%%%%%%%%%%%%%%%%%%%
%%%%%%%%%%%%%

\section{Introduction}
 In this paper, we consider the following unconstrained optimization problem
\begin{equation}\label{obj fun}
	\underset{x \in R^{n}}{min}f(x),
\end{equation}
where the objective function $f:R^{n} \rightarrow R$ is continuously differentiable. The gradient of $f(x)$ is denoted by $g(x)$.

Due to the solid theory and efficient numerical performance, the conjugate gradient(CG) method is a well-known method for solving  unconstrained optimization problem \eqref{obj fun}. Its general iterative formula can be given by
\begin{equation}\label{IM}
	x_{k+1}=x_{k}+\alpha_{k}d_{k},
\end{equation}
\begin{equation}\label{dk}
	d_{k}=\begin{cases}
	-g(x_{0}),& if \text{ } k=0,\\
	-g(x_{k})+\beta_{k}d_{k-1},& if \text{ } k \geq 1,
	\end{cases}
\end{equation}
where $\beta_{k}$ is called conjugate parameter, $d_{k}$ is the search direction, and the step size $\alpha_{k} > 0 $ is obtained by some line search. In this paper, we denote $f_{k}=f(x_k)$, $g_{k}=g(x_k)$.

Different conjugate parameters bring about different CG methods, and their convergence and numerical performance are also greatly variant. Here are several well-known formulas for $\beta_{k}$ such as Fletcher-Reeves(FR) \cite{ref1}, Polak-Ribiere-Polak(PRP) \cite{ref2}, Hestenes-Stiefel(HS) \cite{ref3}, Dai-Yuan(DY) \cite{ref4} and so on \cite{ref5,ref6}.
\begin{align*}
	\beta_{k}^{FR}&=\frac{\Vert g_{k}\Vert^2}{\Vert g_{k-1}\Vert^2}, &\quad \beta_{k}^{PRP}&=\frac{g_{k}^{T}y_{k-1}}{\Vert g_{k-1} \Vert ^2}, & \quad
	\beta_{k}^{HS}&=\frac{g_{k}^{T}y_{k-1}}{d^{T}_{k-1}y_{k-1}},\\
	\beta_{k}^{LS}&=\frac{g_{k}^{T}y_{k-1}}{-d^{T}_{k-1}g_{k-1}}, & \quad
	\beta_{k}^{DY}&=\frac{\Vert g_{k}\Vert^2}{d_{k-1}^{T}y_{k-1}}, & \quad
	\beta_{k}^{P}&=\frac{g_{k}^{T}(y_{k-1}-s_{k-1})}{d_{k-1}^{T}y_{k-1}},
\end{align*}
where $y_{k}=g_{k+1}-g_{k}$,$s_{k}=\alpha_{k}d_{k}=x_{k+1}-x_{k}$.

In order to take advantage of second order information of the objective function, Dai and Kou \cite{ref7} proposed a family of CG methods where the search direction is closest to the direction of the scaled memoryless BFGS method \cite{refp,refs}. The conjugate parameter family is as follows:
\begin{equation}\label{beta DK}
	\beta_{k+1}^{DK}(\tau_{k})=\frac{y_{k}^{T}g_{k+1}}{d_{k}^{T}y_{k}}-(\tau_{k}+\frac{\Vert y_{k} \Vert ^{2}}{s_{k}^{T}y_{k}}-\frac{s_{k}^{T}y_{k}}{\Vert s_{k} \Vert ^{2}})\frac{s_{k}^{T}g_{k+1}}{d_{k}^{T}y_{k}},
\end{equation}
where $\tau_{k}$ is a hyperparameter. Dai and Kou emphasized that $\tau_{k}=\frac{s_{k}^{T}y_{k}}{\Vert s_{k} \Vert ^2}$ is the most efficient choice which corresponds to the conjugate parameter
\begin{equation}\label{beta H}
	\beta_{k+1}^{H}= \frac{y_{k}^{T}g_{k+1}}{d_{k}^{T}y_{k}}-\frac{\Vert y_{k} \Vert ^{2}d_{k}^{T}g_{k+1}}{(d_{k}^{T}y_{k})^{2}}.
\end{equation}
Actually, \eqref{beta H} is also a special case of $\beta_{k}^{\theta}$ from Hager and Zhang's CG\_DESCENT \cite{refzhang}.

%Despite the abundant improvement of the traditional CG methods, there still exist some theoretical and numerical challenges, including the choice of suitable step size and %conjugate parameter, as well as the numerical performance on large-scale problems.

Considering efficient actual performance of the spectral gradient method and inspired by its idea \cite{ref8,ref9}, Brigin and Martinez \cite{ref10} made an effort to combine the CG method with the spectral gradient method and proposed a spectral CG method. The search direction is yielded by
\begin{gather}\label{d theta}
d_{k+1} =-\theta_{k+1}g_{k+1}+\beta_{k+1}d_{k}, \text{ } d_{0} =-g_{0},  \\
\theta_{k+1} =\frac{s_{k}^{T}s_{k}}{s_{k}^{T}y_{k}}, \text{ }
\beta_{k+1} =\frac{(\theta_{k+1}y_{k}-s_{k})^{T}g_{k+1}}{d_{k}^{T}y_{k}}.\notag
\end{gather}
where $\theta_{k+1}$ is called spectral parameter, equal to Barzilai-Borwein(BB) step length in \eqref{d theta}. Brigin and Martinez conducted a large amount of experiments to show that the proposed spectral CG method has a better numerical performance than many traditional CG methods, such as FR, PRP and P. Unfortunately, the descent property of \eqref{d theta} can not be guaranteed.

In 2010, Andrei \cite{ref11} proposed a spectral CG method with sufficient descent property, where the search direction is yielded by
\begin{align}\label{dk thetak}
    &d_{k+1} =-\theta_{k+1}g_{k+1}+\beta_{k+1}^{N}s_{k}, \text{ } d_{0} =-g_{0}, \\
    &\theta_{k+1} =\frac{1}{y_{k}^{T}g_{k+1}}(\Vert g_{k+1}\Vert ^2 -\frac{\Vert g_{k+1}\Vert ^2 s^{T}_{k}g_{k+1}}{y_{k}^{T}s_{k}}), \text{ }
    \beta_{k+1}^{N} =\frac{\Vert g_{k+1}\Vert^2}{y_{k}^{T}s_{k}}-\frac{\Vert g_{k+1}\Vert^2s_{k}^{T}g_{k+1}}{(y_{k}^{T}s_{k})^{2}}.\notag
\end{align}\\
The direction $d_{k+1}$ satisfies that $g_{k+1}^{T}d_{k+1} \leq -(\theta_{k+1}-\frac{1}{4})\Vert g_{k+1}\Vert ^{2} $  which possesses the sufficient descent property in case  $\theta_{k+1}>\frac{1}{4}$. Therefore, to ensure sufficient descent in any case, Andrei reset $\theta_{k+1}=1$  when $\theta_{k+1} \leq \frac{1}{4}$.

In 2017, Jian \cite{ref12} proposed a new spectral CG method. It combines the Dai-Kou conjugate parameter given by \eqref{beta H} with the quasi-Newton direction  $d_{k+1} = -B^{-1}_{k+1}g_{k+1}$, where $B_{k+1}$ is an approximation of the Hessian $\nabla^{2}f(x_{k+1})$. The search direction is as follows:
\begin{align}\label{search direction}
	&d_{k+1}=\theta_{k+1}^{JC}g_{k+1}+\beta_{k+1}^{JC}d_{k},\text{ } d_{0}=-g_{0},\\
	&\theta_{k+1}^{JC}=\begin{cases}
	\theta_{k+1}^{JC+},& if \text{ }\theta_{k+1}^{JC+} \in [\frac{1}{4}+\eta ,\tau] \\
	1,& otherwise
	\end{cases},\quad \notag\\
	&\theta_{k+1}^{JC+}=1-\frac{1}{y_{k}^{T}g_{k+1}}(\frac{\Vert y_{k} \Vert ^{2}d_{k}^{T}g_{k+1}}{d_{k}^{T}g_{k+1}}-s_{k}^{T}g_{k+1}),\quad \notag \\
	&\beta_{k+1}^{JC}=\frac{y_{k}^{T}g_{k+1}}{d_{k}^{T}y_{k}}-\frac{\Vert y_{k}\Vert^{2}d_{k}^{T}g_{k+1}}{(d_{k}^{T}y_{k})^{2}},\notag
\end{align}
where $\eta$ and $\tau$ are positive constants. Although Jian only proved the global convergence of the spectral CG method for uniformly convex functions, the method had a good numerical performance on the public test functions collection \cite{ref13}, even superior to CG\_DESCENT which is one of the most efficient CG methods. In this sense, it is meaningful to proceed with further research on the spectral CG method.

The above mentioned methods from Dai and Kou \eqref{beta DK}, Andrei \eqref{dk thetak} and Jian \eqref{search direction} are all based on the standard secant equation
\begin{equation}\label{function value}
	B_{k+1}s_{k}=y_{k}.
\end{equation}
Note that the standard secant equation employs only gradients but ignores function values. So the $B_{k+1}$ satisfying \eqref{function value} is not a good approximation to the Hessian matrix. Techniques using both gradients and function values have been studied by many researchers. In order to get a better approximation of the Hessian matrix, Yuan \cite{ref15}, Khoshgam \cite{ref16}, Biglari \cite{ref17} and Yuan \cite{ref18} respectively proposed the following modified secant equations, using the
Taylor expansion of the objective function with the interpolation condition.
\begin{equation}\label{zk1}
	B_{k+1}s_{k}=z_{k}^{(\infty)},z_{k}^{(\infty)}=y_{k}+\frac{\mu_{k}^{(\infty)}}{\Vert s_{k} \Vert^{2}}s_{k},
\end{equation}
\begin{equation}\label{zk2}
	B_{k+1}s_{k}=z_{k}^{(5)},z_{k}^{(5)}=y_{k}+\frac{\mu_{k}^{(5)}}{\Vert s_{k} \Vert^{2}}s_{k},
\end{equation}
\begin{equation}\label{zk3}
B_{k+1}s_{k}=z_{k}^{(4)},z_{k}^{(4)}=y_{k}+\frac{\mu_{k}^{(4)}}{\Vert s_{k} \Vert^{2}}s_{k},
\end{equation}
\begin{equation}\label{zk4}
B_{k+1}s_{k}=z_{k}^{(3)},z_{k}^{(3)}=y_{k}+\frac{\mu_{k}^{(3)}}{\Vert s_{k} \Vert^{2}}s_{k},
\end{equation}
where
\begin{gather}
	\mu_{k}^{(\infty)}=2(f_{k}-f_{k+1})+(g_{k}+g_{k+1})^{T}s_{k},\notag \\
	\mu_{k}^{(5)}=\frac{10}{3}(f_{k}-f_{k+1})+\frac{5}{3}(g_{k}+g_{k+1})^{T}s_{k},\notag \\
	\mu_{k}^{(4)}=4(f_{k}-f_{k+1})+2(g_{k}+g_{k+1})^{T}s_{k},\notag \\
	\mu_{k}^{(3)}=6(f_{k}-f_{k+1})+3(g_{k}+g_{k+1})^{T}s_{k}.\notag
\end{gather}

%Apart from the above improvement on the standard secant equation. Li and Fukushima notice that when objective function $f$ is nonconvex, the standard BFGS method with Wolfe-type line search is not necessarily globally convergent. They \cite{ref13} suggested a modified BFGS method, which updates matrix $B_{k+1}$ satisfying \eqref{BFGS} instead of \eqref{function value}
%\begin{align}\label{BFGS}
%	&B_{k+1}s_{k}=\overline{z}_{k},\\
%	&\overline{z}_{k}=y_{k}+h_{k}\Vert g_{k}\Vert^{r}s_{k}, \text{ }
%    h_{k}=D+max \left\lbrace  -\frac{s_{k}^{T}y_{k}}{\Vert s_{k} \Vert^{2}},0  \right\rbrace  \Vert g_{k} \Vert^{-r},\notag
%\end{align}
%where $r$ and $D$ are positive constants, specifically $r=1$ and $D=1$ in \cite{ref13}. They proved the global convergence without the convexity assumption on $f$.
%
%There are many reseachers exerting the modified secant eequation \eqref{BFGS} into CG-type methods to improve convergence. For example, Parvaneh and Keyvan \cite{ref14} replaced $y_{k}$ in \eqref{search direction} with $\overline{z}_{k}$ in \eqref{BFGS} and shown that the modified spectral CG method is globally convergent for nonconvex problems.

Considering that function values contribute to a better approximation to the Hessian matrix, we propose a family of modified secant equations, containing the previous variants \eqref{zk1}-\eqref{zk4}. The idea of the best choice among the modified secant equation family can be interpreted from the perspective of polynomial interpolation.

The standard secant equation \eqref{function value} always satisfies the curvature condition $s_k^T y_k >0$ under the standard Wolfe line search, while the modified ones not necessarily yields $s_k^T z_k^{(m)} >0$, where $m$ is the order of Taylor expansion. To meet the curvature condition, a frequently-used correction \cite{refB, refZ, ref16} is to replace $\mu_{k}^{(m)}$ with $max\{\mu_{k}^{(m)},0\}$. However, from the experimental observation, the negative term in sequence $\{\mu_k^{(m)}\}$ occupies a lager proportion. Thus, simply restricting $\mu_{k}^{(m)}$ to non-negative number is not a proper scheme. In order to not only guarantee the curvature condition holds but also retain the negative term in sequence $\{\mu_k^{(m)}\}$ under some conditions, we develop a modified Wolfe line search.

The paper organized is as follows: In Section 2, we derivate a series of modified secant equations, based on which a family of modified spectral CG methods are obtained. Then we propose a modified Wolfe line search and present a new spectral CG algorithm based on Jian's algorithm. In Section 3, the global convergence characteristics are detailedly analysed for general nonconvex functions. In Section 4, we conduct comparative numerical experiments and plot performance profiles of Dolan-More. Finally, the conclusions are presented in Section 5.\\

\section{A new spectral CG algorithm with a modified Wolfe line search}
\subsection{The modified secant equations and spectral CG methods}

Assume the function  $f$ is smooth enough and consider the mth-order approximate model at the iterate $x_{k+1}$,
\begin{equation}\label{eq:phi}
	\begin{split}
	\varphi_{k+1}^{(m)}(d)&=f_{k+1}+d^{T}g_{k+1}+\frac{1}{2}d^{T}\nabla^{2}f(x_{k+1})d+\frac{1}{3!}d^{T}[\nabla^{3}f(x_{k+1})d]d+\cdots \\
    &+\frac{1}{m!}d^{T}[\nabla^{m}f(x_{k+1})\underbrace{d\cdots d}_{m-2}]d+o(\Vert d \Vert^{m}),
	\end{split}
\end{equation}
where $\varphi_{k+1}^{(m)}(d) \approx f(x_{k+1}+d)$ for all $d \in U(x_{k+1},\epsilon)$, $\nabla^{m}f(x_{k+1}) \in R^{\underbrace{n\times n \times n \times \cdots \times n}_{m}}$ is a mth-order tensor and
\begin{equation}
	d^{T}[\nabla^{m}f(x_{k+1})\underbrace{d\cdots d}_{m-2}]d=\sum_{n_{1},\cdots ,n_{m}=1}^{n}\frac{\partial^{m}f(x_{k+1})}{\partial x^{n_{1}},\cdots,\partial x^{n_{m}}}d^{n_{1}},\cdots,d^{n_{m}}.\notag
\end{equation}
The definition of $\varphi_{k+1}^{(m)}(d)$ in \eqref{eq:phi} implies that
\begin{equation}\label{eq:phik}
	\varphi_{k+1}^{(m)}(0)=f_{k+1},
\end{equation}
\begin{equation}\label{eq:gphi}
	\nabla \varphi_{k+1}^{(m)}(0)=g_{k+1}.
\end{equation}
We hope the following cubic Hermite interpolation conditions on the line between $x_k$ and $x_{k+1}$ are satisfied
\begin{equation}\label{eq:phixk}
	\varphi_{k+1}^{(m)}(x_{k}-x_{k+1})=f_{k},
\end{equation}
\begin{equation}\label{eq:gphixk}
	\nabla \varphi_{k+1}^{(m)}(x_{k}-x_{k+1})=g_{k},
\end{equation}
which are equivalent to
\begin{equation}\label{eq:taylerfxk}
	\begin{split}
	f_{k}=&f_{k+1}-s_{k}^{T}g_{k+1}+\frac{1}{2}s_{k}^{T}\nabla^{2}f(x_{k+1})s_{k}+(-1)^{3}\frac{1}{3!}s_{k}^{T}[\nabla^{3}f(x_{k+1})s_{k}]s_{k}+\cdots+ \\
	 &(-1)^{m}\frac{1}{m!}s_{k}^{T}[\nabla^{m}f(x_{k+1})\underbrace{s_{k}\cdots s_{k}}_{m-2}]s_{k}+o(\Vert s_{k} \Vert ^{m}),
	\end{split}
\end{equation}
\begin{equation}\label{eq:taylorgxk}
	\begin{split}
	g_{k}=&g_{k+1}-\nabla ^{2}f(x_{k+1})s_{k}+(-1)^{2}\frac{1}{2}[\nabla^{3}f(x_{k+1})s_{k}]s_{k}+\cdots+ \\
	&(-1)^{m-1}\frac{1}{(m-1)!}[\nabla^{m}f(x_{k+1})\underbrace{s_{k}\cdots s_{k}}_{m-2}]s_{k}+o(\Vert s_{k} \Vert ^{m-1}).
	\end{split}
\end{equation}
By calculating $m \times f_{k}+s_{k}^{T}g_{k}$, we obtain
\begin{equation}\label{sktnabla}
	\begin{split}
		s_{k}^{T}\nabla^{2}f(x_{x_{k+1}})s_{k}=&s_{k}^{T}y_{k}+\frac{m}{m-2}(2(f_{k}-f_{k+1})+(g_{k}+g_{k+1})^{T}s_{k})+\\
	&(-1)^{2}\frac{m-3}{3(m-2)}s_{k}^{T}[\nabla^{3}f(x_{k+1})s_{k}]s_{k}+\\
	&(-1)^{3}\frac{m-4}{12(m-2)}s_{k}^{T}[\nabla^{4}f(x_{k+1})s_{k}s_{k}]s_{k}+\\
	&\cdots\\
	&(-1)^{m-2}(\frac{(3-m)m}{(m-1)!(m-2)}+\frac{1}{(m-2)!})s_{k}^{T}[\nabla^{m-1}f(x_{k+1})\underbrace{s_{k}\cdots s_{k}}_{m-3}]s_{k}+\\
    &o(\Vert s_{k} \Vert^{m}).
	\end{split}
\end{equation}
Then we get a series of modified secant equations about the order $m$,
\begin{align}\label{eq:Bksk}
	&B_{k+1}s_{k}=z_{k}^{(m)},\text{ } z_{k}^{(m)}=y_{k}+\frac{m\mu_{k}}{(m-2)\Vert s_{k} \Vert ^{2}}s_{k},\text{ } m\geq 3, \text{ } m\in Z,\\
	&\mu_{k}=2(f_{k}-f_{k+1})+(g_{k}+g_{k+1})^{T}s_{k},\notag
\end{align}
all of which are appropriate approximation to $\nabla^{2}f(x_{k+1})$. Putting $m=3,4,5$  in \eqref{eq:Bksk} yields the secant equations \eqref{zk4}, \eqref{zk3} and \eqref{zk2}. It is notable that \eqref{zk1} can be regarded as \eqref{eq:Bksk} with $m=+\infty$ , because $\lim\limits_{m \to \infty ,m \in Z} \frac{m}{m-2} =1$. In the following theorem, we discuss the properties that are distinguishable among modified secant equations given by \eqref{eq:Bksk}.
\begin{theorem}
	Assume the function  $f$ is smooth enough. If $\Vert s_{k} \Vert$ is sufficiently small, then we have
	\begin{gather}
		s_{k}^{T}\nabla^{2}f(x_{k+1})s_{k}-s_{k}^{T}z_{k}^{(\infty)}=\frac{1}{3}s_{k}^{T}[\nabla^{3}f(x_{k+1})s_{k}]s_{k}+o(\Vert s_{k} \Vert^{3}),\notag \\
		s_{k}^{T}\nabla^{2}f(x_{k+1})s_{k}-s_{k}^{T}z_{k}^{(5)}=\frac{2}{9}s_{k}^{T}[\nabla^{3}f(x_{k+1})s_{k}]s_{k}+o(\Vert s_{k} \Vert^{3}),\notag \\
		s_{k}^{T}\nabla^{2}f(x_{k+1})s_{k}-s_{k}^{T}z_{k}^{(4)}=\frac{1}{6}s_{k}^{T}[\nabla^{3}f(x_{k+1})s_{k}]s_{k}+o(\Vert s_{k} \Vert^{3}),\notag \\
		s_{k}^{T}\nabla^{2}f(x_{k+1})s_{k}-s_{k}^{T}z_{k}^{(3)}=o(\Vert s_{k} \Vert^{3}),\notag \\
		s_{k}^{T}\nabla^{2}f(x_{k+1})s_{k}-s_{k}^{T}z_{k}^{(m)}=\frac{m-3}{3(m-2)}s_{k}^{T}[\nabla^{3}f(x_{k+1})s_{k}]s_{k}+o(\Vert s_{k} \Vert ^{3}).\notag
	\end{gather}
\end{theorem}\label{thm:fsmooth}
\begin{proof}
	 See the equality \eqref{sktnabla}.
\end{proof}
\begin{remark}
 We give an interpretation of the above theoretical results by polynomial interpolation. For the modified secant equation \eqref{zk4}, Yuan \cite{ref18} supposed a cubic approximation to the objective function and used two-point cubic Hermite interpolation conditions \eqref{eq:phik}-\eqref{eq:gphixk} at $x_{k}$  and $x_{k+1}$, hence the $s_{k}^{T}z_{k}^{(3)}$ more exactly approaches to $s_{k}^{T}\nabla^{2}f(x_{k+1})s_{k}$. For \eqref{zk3} and \eqref{zk2}, Biglari \cite{ref17} and Khoshgam \cite{ref16} considered the quartic and quintic model approximating to the function $f$, but they still simply employed cubic Hermite interpolation conditions. It is known that it takes at least $m+1$ interpolation conditions to determine a mth-order interpolation function. That means only interpolation conditions \eqref{eq:phik}-\eqref{eq:gphixk} are insufficient for a higher-order ($m>3$) model, which results in $s_{k}^{T}z_{k}^{(5)}$  and $s_{k}^{T}z_{k}^{(4)}$  possessing low approximation quality compared to $s_{k}^{T}z_{k}^{(3)}$. It is also intuitive that larger value of $m$ leads to the four conditions \eqref{eq:phik}-\eqref{eq:gphixk} appearing more insufficient and causes lower approximation accuracy. As for \eqref{zk1}, although it is actually from the quadratic model with interpolation conditions \eqref{eq:phik}-\eqref{eq:phixk}, on the other hand, it can be yielded from the model where $m=\infty$ with the conditions \eqref{eq:phik}-\eqref{eq:gphixk}. Thus the $s_{k}^{T}z_{k}^{(\infty)}$ is naturally in last place.
\end{remark}

In the following, we improve the spectral CG method proposed by Jian. Consider the following spectral conjugate direction
\begin{equation}\label{eq:dk+1}
	d_{k+1}=-\theta_{k+1}g_{k+1}+\beta_{k+1}d_{k},\text{ } d_{0}=-g_{0},
\end{equation}
where the conjugate parameter $\beta_{k+1}$ is given by a truncated form of \eqref{beta H},
\begin{align}\label{eq:beta k+1}
	&{{\beta }_{k+1}^{M}}=\max \{{\beta }_{k+1}^{L}, {\beta}_{k+1}^{R} \}, \\
    & {\beta }_{k+1}^{L}=\frac{g_{k+1}^{T}{{z}_{k}^{(m)}}}{d_{k}^{T}{{z}_{k}^{(m)}}}-\frac{||{{z}_{k}^{(m)}}|{{|}^{2}}}{d_{k}^{T}{{z}_{k}^{(m)}}}\frac{g_{k+1}^{T}{{d}_{k}}}{d_{k}^{T}{{z}_{k}^{(m)}}}, \notag\\
    & {\beta}_{k+1}^{R}=\frac{g_{k}^{T}{{d}_{k}}}{||{{d}_{k}}|{{|}^{2}}}.\notag
\end{align}

As for the choice for the spectral parameter $\theta_{k+1}$, we hope $d_{k+1}$ could be the quasi-Newton direction, i.e.
\begin{equation}\label{eq:d_k+1}
	d_{k+1} = -B_{k+1}^{-1}g_{k+1}.
\end{equation}
It follows that
\begin{equation}\label{eq:Bk+1}
	-B_{k+1}^{-1}g_{k+1}= -\theta_{k+1}g_{k+1}+\beta_{k+1}d_{k}.
\end{equation}
Multiplying \eqref{eq:Bk+1} by $s_k^{T}B_{k+1}$, we have
\begin{equation}\label{eq:theta}
	\theta_{k+1} = \frac{1}{s_{k}^{T}B_{k+1}g_{k+1}}(s_{k}^{T}g_{k+1}+\beta_{k+1}d_{k}^{T}B_{k+1}s_{k}).
\end{equation}
We apply the modified secant condition $B_{k+1}s_{k}=z_{k}^{(m)}$ to \eqref{eq:theta} to obtain a choice for $\theta_{k+1}$,
\begin{equation}
\tilde{\theta}_{k+1}= \frac{1}{g_{k+1}^{T}z_{k}^{(m)}}(s_{k}^{T}g_{k+1}+\beta_{k+1}d_{k}^{T}z_{k}^{(m)}).
\end{equation}
The following theorem \cite{ref12} indicates that the spectral
conjugate direction ${d}_{k+1}$ with ${\theta}_{k+1}=\tilde{\theta}_{k+1}, \beta_{k+1}=\beta_{k+1}^{L}$ has sufficient descent property under certain condition.
\begin{theorem}\label{thm:dk+1} \cite{ref12}
	If $g_{k+1}^{T}z_{k}^{(m)} \neq 0$ and $d_{k}^{T}z_{k}^{(m)} \neq 0$, the direction ${d}_{k+1}$, where ${\theta}_{k+1}=\tilde{\theta}_{k+1}, \beta_{k+1}=\beta_{k+1}^{L}$, satisfies
	\begin{equation}\label{eq:dksatisify}
		g_{k+1}^{T}{d}_{k+1} \leq -({\theta}_{k+1}-\frac{1}{4})\Vert g_{k+1} \Vert^{2},\text{ } k \geq 0.
	\end{equation}
	where ${\theta}_{1}=1$. If ${\theta}_{k+1}>\frac{1}{4}$, ${d}_{k+1}$is a sufficient descent direction.
\end{theorem}
Further, to ensure sufficient descent of search directions and boundedness of spectral parameters, we take a truncation for $\tilde{\theta}_{k+1}$:
\begin{equation}\label{eq:thetak+1}
	\bar{\theta}_{k+1}=\begin{cases}
	\tilde{\theta}_{k+1},& if \text{ } \tilde{\theta}_{k+1} \in [\frac{1}{4}+\eta ,\tau],\\
	1,& otherwise,
	\end{cases}
\end{equation}
where $\eta$  is a small positive constant and $\tau$ is a upper bound, leading to $\bar{\theta}_{k+1} \in [\frac{1}{4}+\eta ,\tau]$ .
\begin{corollary}\label{thm:2.3}
	 From $\theta_{k+1}=\bar{\theta}_{k+1}$  and \eqref{eq:dksatisify}, we have
	 \begin{equation}\label{eq:gkgeq}
	 	g_{k+1}^{T}{d}_{k+1} \leq -\eta \Vert g_{k+1} \Vert^{2},\text{ } k \geq 0.
	 \end{equation}
\end{corollary}
Then, we show ${{d}_{k+1}}$ with $\beta_{k+1},\theta_{k+1}$ given by \eqref{eq:beta k+1}, \eqref{eq:thetak+1} possesses sufficient descent property.
\begin{theorem}\label{thm:ndk+1}
	 For the direction ${d}_{k+1}$ where ${\theta}_{k+1}=\tilde{\theta}_{k+1}, \beta_{k+1}=\beta_{k+1}^{M}$, if $g_{k+1}^{T}z_{k}^{(m)} \neq 0$ and $d_{k}^{T}z_{k}^{(m)} \neq 0$, we have
	 \begin{equation}\label{eq:gdec}
	 	g_{k+1}^{T}d_{k+1} \leq -\eta \Vert g_{k+1} \Vert^{2},\text{ } k \geq 0.
	 \end{equation}
\end{theorem}
\begin{proof}
	 Use mathematical induction. First, $d_0$ satisfies that $g_0^{T}d_0=-\Vert g_0 \Vert^2 < -\eta \Vert g_0 \Vert^2$. Then, suppose $d_k$ satisfies \eqref{eq:gdec}. If ${{\beta }_{k+1}}= {\beta }_{k+1}^{L}$, then \eqref{eq:gdec} follows from Corollary \ref{thm:2.3}. If ${{\beta }_{k+1}}= {\beta }_{k+1}^{R}$ and ${{d}_{k}}$ satisfies \eqref{eq:gdec}, then ${\beta }_{k+1}^{L} \leq {\beta }_{k+1}^{R} < 0$. Multiply ${{d}_{k+1}}$ by ${g}_{k+1}^{T}$, we have
\begin{equation}
	g_{k+1}^{T}d_{k+1} = -{\theta}_{k+1} \Vert g_{k+1} \Vert^{2} + {\beta }_{k+1}^{R}g_{k+1}^{T}d_{k}. \notag
\end{equation}
 If $g_{k+1}^{T}d_{k} \ge 0$, then \eqref{eq:gdec} is obtained immediately. If $g_{k+1}^{T}d_{k} < 0$, then
\begin{equation}
	g_{k+1}^{T}d_{k+1} = -{\theta}_{k+1} \Vert g_{k+1} \Vert^{2} + {\beta }_{k+1}^{R}g_{k+1}^{T}d_{k} \leq -{\theta}_{k+1} \Vert g_{k+1} \Vert^{2} + {\beta }_{k+1}^{L}g_{k+1}^{T}d_{k}. \notag
\end{equation}
\eqref{eq:gdec} follows from Corollary \ref{thm:2.3}. Hence, ${{d}_{k+1}}$ satisfies \eqref{eq:gdec}. The proof is complete.
\end{proof}

\subsection{A modified Wolfe line search}

We notice that for nonlinear functions, the value $s_{k}^{T}z_{k}^{(m)}$ may turn out to be zero or negative while $s_{k}^{T}y_{k}$ is always positive. Thus, \eqref{eq:Bksk} seems not to be reasonable. The most common improved measure is to replace $\mu_{k}$ in \eqref{eq:Bksk} with $\mu_{k}^{+}=max\{\mu_{k},0\}$, which is only a compromised method. When some $\mu_{k}$ is negative, we have to discard it and the modified secant equation \eqref{eq:Bksk} degenerates into the standard secant equation \eqref{function value}.

From the experimental observations, where we take ARWHEAD function for example and display results by Figure \ref{sgn} and Table \ref{24}, we find that the negative value in sequence $\{\mu_k\}$ occupies a not small proportion; The absolute value of $\mu_{k}$ is big in the only first few iterations, where $\mu_{k}$ is nearly always negative. However, the positive $\mu_{k}$ is almost always very small. Thus, it is necessary to take into account negative $\mu_{k}$. The purpose of this section is to develop an modified Wolfe line search, which enables negative $\mu_{k}$ to be retained, not to be compulsorily set as 0.

\begin{figure}[htbp]
	\centering
	\includegraphics[scale=0.6]{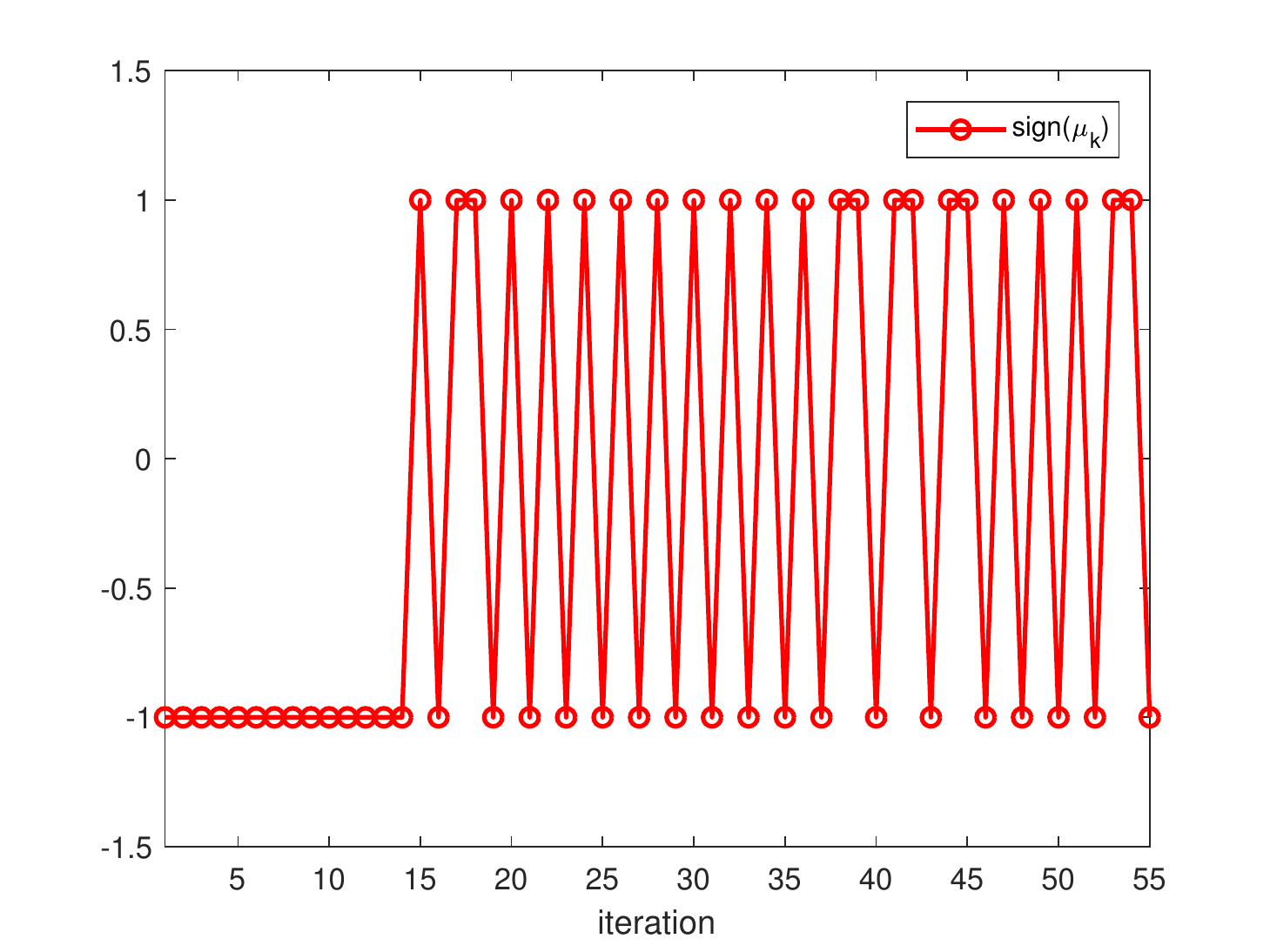}	
	\caption{The sign of $\mu_k$}
    \label{sgn}	
\end{figure}

\begin{table}[H]
\tiny
\begin{tabular}{|c|c|c|c|c|c|c|c|c|}
\hline
iteration & 1          & 2          & 3          & 4          & 5          & 6          & 7          & 8          \\ \hline
$\mu_k$         & -1.107E+05 & -2.634E+03 & -3.307E+03 & -2.750E-01 & -8.784E-01 & -1.102E-07 & -4.256E-07 & -6.419E-04 \\ \hline
iteration & 9          & 10         & 11         & 12         & 13         & 14         & 15         & 16         \\ \hline
$\mu_k$         & -2.595E-09 & -2.381E-07 & -5.132E-14 & -7.026E-13 & -2.812E-13 & -3.412E-13 & 2.983E-14  & -3.223E-14 \\ \hline
iteration & 17         & 18         & 19         & 20         & 21         & 22         & 23         & 24         \\ \hline
$\mu_k$         & 7.043E-16  & 2.123E-14  & -2.165E-14 & 1.962E-13  & -1.958E-13 & 4.604E-13  & -4.602E-13 & 5.752E-13  \\ \hline
\end{tabular}
\caption{The value of $\mu_k$ in former 24 iterations}
\label{24}
\end{table}

Before giving the modified Wolfe line search, we need a limitation upon $z_{k}^{(m)}$ in \eqref{eq:Bksk},
\begin{align}\label{eq:Bkskl}
	&z_{k}^{(m)}=y_{k}+t_{k}s_{k},\text{ } m\geq 3, \text{ } m\in Z,\\
    &t_k=\begin{cases}
        \frac{m\mu_{k}}{(m-2)\Vert s_{k} \Vert ^{2}}, &if \text{ } \mu_{k} > 0,\\
        C\frac{\mu_{k}}{\Vert s_{k} \Vert ^{2}}, &if \text{ } \mu_{k} \le 0,
    \end{cases}\notag\\
	&\mu_{k}=2(f_{k}-f_{k+1})+(g_{k}+g_{k+1})^{T}s_{k},\notag\\
    &C=\frac{\sigma-\rho}{1-2\rho+\sigma}, 0<\rho <\sigma <1.\notag
\end{align}
Now, based on the standard Wolfe line search, our modified Wolfe line search is given as follows:
\begin{equation}\label{eq:wolfe conditon}
	\begin{split}
	&f(x_{k}+\alpha_{k}d_{k}) \leq f(x_{k})+\rho \alpha_{k}g_{k}^{T}d_{k},\\
	&(g(x_{k}+\alpha_{k}d_{k})+\min \{t_{k},0\}s_{k})^{T}d_{k} \geq \sigma g_{k}^{T}d_{k},
	\end{split}
\end{equation}

It is not difficult to prove that there exists a suitable stepsize satisfying \eqref{eq:wolfe conditon}. First, some basic and mild assumptions as follows are necessary.
\begin{assum}\label{p:3.1}
(1) $f$ is bounded below, i.e. $f(x)>-\infty,\text{ }\forall x \in R^{n}$.\\
(2) The level set $S=\{x\in {{R}^{n}}:f(x)\le f({{x}_{0}})\}$is bounded, namely, there exists a constant $M$ such that
\begin{equation}\label{eq:51}
	||x||\le M,\text{ }\forall x\in S.
\end{equation}
(3) In some neighborhood $\Omega $ of $S$, $f$ is continuously differentiable and its gradient is Lipschitz continuous, i.e. there exists a constant $L>0$ such that
\begin{equation}\label{eq:52}
	||g(x)-g(y)||\le L||x-y||,\text{ }\forall x,y\in \Omega,
\end{equation}
which implies that there exists a constant $\gamma > 0$ such that
\begin{equation}\label{eq:53}
	||g(x)||\le \gamma ,\text{ }\forall x\in S.
\end{equation}
\end{assum}
\begin{lemma}\label{lemma:exist}
    Suppose that Assumption \ref{p:3.1} holds and $d_k$ is a descent direction, then there must exist a suitable stepsize satisfying modified Wolfe conditons \eqref{eq:wolfe conditon}.
    \begin{proof}
        It is known that there exists a smallest stepsize $\alpha_k^{\prime}>0$ such that
        \begin{equation}\label{2.22}
            f(x_{k}+\alpha_k^{\prime}d_k)=f(x_k)+\alpha_k^{\prime}\rho g_k^{T}d_k.
        \end{equation}
        The relation $f(x_{k}+\alpha_{k}d_{k}) \leq f(x_{k})+ \rho \alpha_{k}g_{k}^{T}d_{k}$ clearly holds for all stepsizes less than $\alpha_k^{\prime}$. By the mean value theorem, there exists $\alpha_k^{\prime\prime} \in (0,\alpha_k^{\prime})$ such that
        \begin{equation}\label{2.23}
            f(x_{k}+\alpha_k^{\prime}d_k)-f(x_k)=\alpha_k^{\prime}g(x_k+\alpha_k^{\prime\prime}d_k)^{T}d_k.
        \end{equation}
        Combine \eqref{2.22} with \eqref{2.23}, we obtain
        \begin{equation}
            g(x_k+\alpha_k^{\prime\prime}d_k)^{T}d_k=\rho g_k^{T}d_k>\sigma g_k^{T}d_k.
        \end{equation}
        Let $s_k=\alpha_k^{\prime\prime}d_{k}$. If $\mu_k>0$, it is obvious that the second condition in \eqref{eq:wolfe conditon} holds. If $\mu_k \le 0$, we have
        \begin{equation}
        \begin{split}
            g(x_k+\alpha_k^{\prime\prime}d_k)^{T}s_k+t_{k}||s_{k}||^2 &=\rho g_k^{T}s_k+C\mu_k=\rho g_k^{T}s_k+C(2(f_{k}-f_{k+1})+(g_{k}+g_{k+1})^{T}s_{k})\\
            & \ge \rho g_k^{T}s_k+C(-2\rho g_k^{T}s_k+g_k^{T}s_k+\sigma g_k^{T}s_k)\\
            & =\rho g_k^{T}s_k+(\sigma-\rho)g_k^{T}s_k\\
            & =\sigma g_k^{T}s_k.
        \end{split}
        \end{equation}
        So the second condition in \eqref{eq:wolfe conditon} holds. The proof is complete.
    \end{proof}
\end{lemma}

For the modified Wolfe line search, we can establish the Zoutendijk condition\cite{ref19}, which is a significant tool to proof the convergence of CG-type methods.
\begin{lemma}\label{lemma:2.7}
	Suppose that Assumption \ref{p:3.1} holds. Consider the iterative scheme \eqref{IM}, where ${{d}_{k}}$ is a descent direction and ${{\alpha }_{k}}$ is determined by the modified Wolfe conditions. Then
	\begin{equation}\label{eq:54}
		\sum\limits_{k=0}^{\infty }{\frac{{{(g_{k}^{T}{{d}_{k}})}^{2}}}{||{{d}_{k}}|{{|}^{2}}}<\infty }.
	\end{equation}
\end{lemma}
\begin{proof}
Note that
\begin{equation}\label{eq:222}
    \begin{split}
    {\mu}_{k}& =2(f_{k}-f_{k+1})+(g_{k}+g_{k+1})^{T}s_{k}=2g^{T}({\varepsilon}_{k})(x_{k}-x_{k+1})+(g_{k}+g_{k+1})^{T}s_{k} \\
    & =-2g^{T}({\varepsilon}_{k})s_{k}+(g_{k}+g_{k+1})^{T}s_{k}=(g_{k}-g^{T}({\varepsilon}_{k})+g_{k+1}-g^{T}({\varepsilon}_{k}))^{T}s_{k}
    \end{split}
\end{equation}
where ${{\varepsilon }_{k}}=v{{x}_{k}}+(1-v){{x}_{k+1}},v\in (0,1)$. By \eqref{eq:52}, we have
\begin{equation}\label{eq:63}
	\begin{split}
	|{{\mu }_{k}}|&\le (||{{g}_{k}}-g({{\varepsilon }_{k}})||+||{{g}_{k+1}}-g({{\varepsilon }_{k}})||)||{{s}_{k}}|| \\
	&\le L(||{{x}_{k}}-{{\varepsilon }_{k}}||+||{{\varepsilon }_{k}}-{{x}_{k+1}}||)||{{s}_{k}}||=L||{{s}_{k}}|{{|}^{2}}. \\
	\end{split}
\end{equation}
If $\mu_k>0$, then $0<t_k \le \frac{mL}{m-2}$. If $\mu_k \le 0$, then $-CL \le t_k \le 0$, where $0<C<1$. Thus we have
	\begin{equation}\label{tk}
		-CL \le t_k \le \frac{mL}{m-2}.
	\end{equation}
It follows from \eqref{eq:52} and the second condition in \eqref{eq:wolfe conditon} that
	\begin{equation}
		L\alpha_{k} ||d_k||^2 \ge (g_{k+1}-g_{k})^{T}d_{k} \ge (\sigma-1)g_k^{T}d_k-t_{k} \alpha_{k}||d_k||^2, \notag
	\end{equation}
which with \eqref{tk} gives
	\begin{equation}\label{2.32}
		\alpha_{k} \ge \frac{\sigma-1}{L+t_k} \frac{g_k^T d_k}{||d_k||^2} \ge \frac{\sigma-1}{L(1+\frac{m}{m-2})} \frac{g_k^T d_k}{||d_k||^2}.
	\end{equation}
It follows from \eqref{2.32} and the first condition in \eqref{eq:wolfe conditon} that
	\begin{equation}\label{2.34}
		f_k-f_{k+1}\ge p \frac{(g_k^{T}d_k)^2}{||d_k||^2},
	\end{equation}
where $p=\frac{\rho(1-\sigma)}{L(1+\frac{m}{m-2})}$. Summing \eqref{2.34} over $k$ and noting that $f$ is bounded below, we see \eqref{eq:54} holds. The proof is complete.
\end{proof}
Now, based on the above analysis, the detailed steps of the new algorithm, denoted SCGMMWLS, is given below
\begin{algorithm}[htb]
	\caption{ SCGMMWLS}
	\label{alg:Framwork1}
	\begin{algorithmic}[1]
		\Require
		An initial point ${{x}_{0}}\in {{R}^{n}}$, $\varepsilon >0$, positive integer $MaxIter$.
		\State Set $k=0,\text{ }{{d}_{0}}=-{{g}_{0}}$.
		\State If $\Vert{{g}_{k}}\Vert _{\infty }\leq \varepsilon $ or $k> MaxIter$, stop.
		\State Compute $\alpha_{k}$  by the modified
		Wolfe conditions \eqref{eq:wolfe conditon}.
		\State Update ${{x}_{k+1}}={{x}_{k}}+{{\alpha }_{k}}{{d}_{k}}$ and $g_{k+1}$.
		\State Compute $\beta _{k+1}$ ,  $\theta_{k+1}$ and  $d_{k+1}$ by \eqref{eq:beta k+1}, \eqref{eq:thetak+1} and \eqref{eq:dk+1}, where $z_{k}^{(m)}$ is determined by \eqref{eq:Bkskl}.
		\State Set  $k=k+1$ and go to step 2.
	\end{algorithmic}
\end{algorithm}

\section{The analysis of convergence}

In this section, we show that SCGMMSE has global convergence that
\begin{equation}\label{eq:58}
	\underset{k\to \infty }{\mathop{\lim }}\,\inf ||{{g}_{k}}||=0.
\end{equation}
To this aim, we proceed by contradiction and assume that \eqref{eq:58} does not hold, i.e. there exists a positive constant $\xi $ such that
\begin{equation}\label{eq:59}
	||{{g}_{k}}||\ge \xi,\text{ }k\ge 0.
\end{equation}

Besides Assumption \ref{p:3.1} and lemma \ref{lemma:2.7}, here are two necessary lemmas.
\begin{lemma}\label{lem:3.6}
	Suppose Assumption \ref{p:3.1} holds. Consider the iterative form \eqref{IM}, where ${{\alpha }_{k}}$ is determined by the modified Wolfe conditions and $d_k$ is given by \eqref{eq:dk+1} containing $\beta _{k+1}$ \eqref{eq:beta k+1} and $\theta_{k+1}$ \eqref{eq:thetak+1}. If \eqref{eq:59} holds, then
	\begin{equation}\label{eq:67}
		\sum\limits_{k\ge 1}{||{{u}_{k}}-{{u}_{k-1}}|{{|}^{2}}<\infty },\text{ }{{u}_{k}}=\frac{{{d}_{k}}}{||{{d}_{k}}||},\text{ }{{d}_{k}}\ne 0.
	\end{equation}
\end{lemma}
\begin{proof}
	Divide ${{\beta }_{k+1}}$ given by \eqref{eq:beta k+1} into two parts as follows:
	\begin{equation}\label{eq:68}
		%\beta _{k+1}^{1}=\max \{\frac{g_{k+1}^{T}{{v}_{k}}}{d_{k}^{T}{{v}_{k}}}-(1+\frac{{{(s_{k}^{T}{{v}_{k}})}^{2}}}{||{{s}_{k}}|{{|}^{2}}||{{v}_{k}}|{{|}^{2}}})\frac{||{{v}_{k}}|{{|}^{2}}g_{k+1}^{T}{{d}_{k}}}{{{(d_{k}^{T}{{v}_{k}})}^{2}}}+(1-\mu )\frac{g_{k+1}^{T}{{d}_{k}}}{||{{d}_{k}}|{{|}^{2}}},0\},
    \beta _{k+1}^{1}=\max \{\beta _{k+1},0\},
	\end{equation}
	\begin{equation}\label{eq:69}
		%\beta _{k+1}^{2}=\mu \frac{g_{k+1}^{T}{{d}_{k}}}{||{{d}_{k}}|{{|}^{2}}}.
    \beta _{k+1}^{2}=\min \{\beta_{k+1},0\}.
	\end{equation}
	Define
	\begin{equation}\label{eq:70}
		{{\omega }_{k}}=\frac{-{{\theta }_{k}}{{g}_{k}}+\beta _{k}^{2}{{d}_{k-1}}}{||{{d}_{k}}||},
	\end{equation}
	\begin{equation}\label{eq:71}
		{{\delta }_{k}}=\frac{\beta _{k}^{1}||{{d}_{k-1}}||}{||{{d}_{k}}||}.
	\end{equation}
	From ${{d}_{k+1}}=-{{\theta }_{k+1}}{{g}_{k+1}}+{{\beta }_{k+1}}{{d}_{k}}$, we have\
	\begin{equation}\label{eq:72}
		{{u}_{k}}={{\omega }_{k}}+{{\delta }_{k}}{{u}_{k-1}},\text{ }k\ge 1.
	\end{equation}
	%Using the identity $||{{u}_{k}}||=||{{u}_{k-1}}||=1$ and \eqref{eq:72}, we obtain
    Since $\{u_k,k \ge 0 \}$ are unit vectors,
	\begin{equation}\label{eq:73}
		||{{\omega }_{k}}||=||{{u}_{k}}-{{\delta }_{k}}{{u}_{k-1}}||=||{{\delta }_{k}}{{u}_{k}}-{{u}_{k-1}}||.
	\end{equation}
	Using \eqref{eq:73} and the condition ${{\delta }_{k}}\ge 0$, we derivate
	\begin{equation}\label{eq:74}
		\begin{split}
		& ||{{u}_{k}}-{{u}_{k-1}}||\le ||(1+{{\delta }_{k}}){{u}_{k}}-(1+{{\delta }_{k}}){{u}_{k-1}}|| \\
		& \text{                }\le ||{{u}_{k}}-{{\delta }_{k}}{{u}_{k-1}}||+||{{\delta }_{k}}{{u}_{k}}-{{u}_{k-1}}||=2||{{\omega }_{k}}||. \\
		\end{split}
	\end{equation}
	By \eqref{eq:59}, the relation ${{\theta }_{k}}\le \tau $, the definition of $\beta _{k+1}^{2}$ and the fact $\frac{g_{k}^{T}{{d}_{k}}}{||{{d}_{k}}|{{|}^{2}}} \leq \beta _{k+1} \leq 0$, we have
	\begin{equation}\label{eq:75}
		||-{{\theta }_{k}}{{g}_{k}}+\beta _{k}^{2}{{d}_{k-1}}||\le \tau ||{{g}_{k}}||+\frac{|g_{k-1}^{T}{{d}_{k-1}}|}{||{{d}_{k-1}}|{{|}^{2}}}||{{d}_{k-1}}||\le(1 +\tau )\gamma.
	\end{equation}
	From \eqref{eq:70}, \eqref{eq:74} and \eqref{eq:75}, it follows that
	\begin{equation}\label{eq:76}
		||{{u}_{k}}-{{u}_{k-1}}||\le \frac{2(1 +\tau )\gamma}{||{{d}_{k}}||}.
	\end{equation}
	The descent property of ${{d}_{k}}$ \eqref{eq:gdec} with the relation \eqref{eq:59} yields
	\begin{equation}\label{eq:77}
		\frac{1}{||{{d}_{k}}|{{|}^{2}}}\le \frac{1}{{{\xi }^{4}}}\frac{||{{g}_{k}}|{{|}^{4}}}{||{{d}_{k}}|{{|}^{2}}}\le \frac{1}{{{\xi }^{4}}{{\eta }^{2}}}\frac{{{(g_{k}^{T}{{d}_{k}})}^{2}}}{||{{d}_{k}}|{{|}^{2}}}.
	\end{equation}
	Thus \eqref{eq:67} holds by Lemma \ref{lemma:2.7}. The proof is complete.
\end{proof}

\begin{lemma}\label{lemma:3.7}
Suppose Assumption \ref{p:3.1} holds. Consider the iterative form \eqref{IM}, where ${{\alpha }_{k}}$ is determined by the modified Wolfe conditions and $d_k$ is given by \eqref{eq:dk+1} containing $\beta _{k+1}$ \eqref{eq:beta k+1} and $\theta_{k+1}$ \eqref{eq:thetak+1}, then ${{\beta }_{k+1}}$ has the Property (*) in \cite{ref20}:\\
(1) There exists a constant $b>1$, such that $|{{{\beta}}_{k+1}}|\le b,\text{ }\forall k\ge 0,$\\
(2) There exists a constant $c>0$, such that $||{{s}_{k}}||\le c$, then $|{{{\beta }}_{k+1}}|\le \frac{1}{b},\text{ }\forall k\ge 0.$
\end{lemma}

\begin{proof}
	By the definition of ${{\beta }_{k+1}}$ in \eqref{eq:beta k+1}, we have
	\begin{equation}
    \begin{split}
		\beta_{k+1}=\beta_{k+1}^{L},\text{ } if\text{ }\beta_{k+1}^{L} \ge 0, \notag\\
        0>\beta_{k+1} \ge \beta_{k+1}^{L},\text{ } if\text{ }\beta_{k+1}^{L} < 0.
    \end{split}
	\end{equation}
    Hence, $|\beta_{k+1}| \le |\beta_{k+1}^{L}|$. From \eqref{eq:gdec}, \eqref{eq:59} and the second condition in \eqref{eq:wolfe conditon}, it follows that
	\begin{equation}\label{eq:78}
		d_{k}^{T}{{z}_{k}^{(m)}}\ge -(1-\sigma )g_{k}^{T}d_{k} \ge \eta (1-\sigma)||g_k||^2 \ge \eta (1-\sigma)\xi^2.
	\end{equation}
    If $\mu_k>0$, it is easy to obtain
	\begin{equation}\label{3.15}
        |\frac{d_k^{T}g_{k+1}}{d_k^{T}z_k^{m}}| \le max \left\lbrace \frac{\sigma}{1-\sigma},1 \right\rbrace.
	\end{equation}
    If $\mu_k \le 0$, we have
	\begin{equation}\label{3.16}
        g_{k+1}^{T}d_k=d_k z_k^{m}-(-g_k^{T}d_k+t_k s_k^{T}d_k)<d_k z_k^{m}.
	\end{equation}
    The inequality above is due to that
	\begin{equation}
    \begin{split} \notag
        -g_k^{T}d_k+t_k s_k^{T}d_k &=\frac{1}{\alpha_{k}}( -g_k^{T}s_k+C\mu_{k}) \\ &=\frac{1}{\alpha_{k}}(-g_k^{T}s_k+C(2(f_{k}-f_{k+1})+(g_{k}+g_{k+1})^{T}s_{k}))\\ & \ge\frac{1}{\alpha_{k}}(-g_k^{T}s_k+C(-2\rho g_k^{T}s_k+g_k^{T}s_k+\sigma g_k^{T}s_k))\\
        & =\frac{1}{\alpha_{k}}(-g_k^{T}s_k+(\sigma-\rho)g_k^{T}s_k)>0.
	\end{split}
    \end{equation}
    Also, the second condition in \eqref{eq:wolfe conditon} gives
	\begin{equation}\label{3.17}
        \frac{d_k^{T}g_{k+1}}{d_k^{T}z_k^{m}} \ge \frac{\sigma-\frac{C\mu_k}{g_k^{T}s_k}}{\sigma-1} \ge \frac{\rho}{\sigma-1}.
	\end{equation}
    The second inequality above is from that $\frac{C\mu_k}{g_k^{T}s_k} \ge \sigma - \rho$.
	By \eqref{3.16} and \eqref{3.17}, we have
	\begin{equation}\label{3.18}
		|\frac{d_{k}^{T}{{g}_{k+1}}}{d_{k}^{T}{{z}_{k}^{(m)}}}| \le max \left\lbrace \frac{\rho}{1-\sigma},1 \right\rbrace,
	\end{equation}
    which with \eqref{3.15} implies
	\begin{equation}\label{3.19}
		|\frac{d_{k}^{T}{{g}_{k+1}}}{d_{k}^{T}{{z}_{k}^{(m)}}}| \le max \left\lbrace \frac{\sigma}{1-\sigma},1 \right\rbrace,
	\end{equation}
	Using \eqref{eq:52} and \eqref{eq:63}, we obtain
	\begin{equation}\label{eq:80}
		||{z}_{k}^{(m)}||\le ||y_k||+ \max \left\lbrace C,\frac{m}{m-2} \right\rbrace \frac{|{{\mu }_{k}}|}{||{{s}_{k}}|{{|}^{2}}}||{{s}_{k}}||\le \frac{(1+L)m}{m-2}||{{s}_{k}}||.
	\end{equation}
	It can be shown that together with \eqref{eq:51}, \eqref{eq:53}, \eqref{eq:78}, \eqref{3.19} and \eqref{eq:80},
	$$|{{{\beta }}_{k+1}}|\le \bar{c}||{{s}_{k}}||,\forall k\ge 0, $$
	where $\bar{c}$ is a positive constant. Let $b=\max \{1,2\bar{c}M\},\text{ }c=\frac{1}{b\bar{c}}$, then for all $k\ge 0$,
	\begin{equation}\label{eq:82}
		|{{{\beta }}_{k+1}}|\le b,
	\end{equation}
	\begin{equation}\label{eq:83}
		||{{s}_{k}}||\le c\Rightarrow |{{{\beta }}_{k+1}}|\le \frac{1}{b}.
	\end{equation}
	The relations \eqref{eq:82} and \eqref{eq:83} indicate that ${{\beta}_{k+1}}$   has the Property (*) in \cite{ref20}. The proof is complete.
\end{proof}

Based on the above lemmas, the global convergence of SCGMMSE is established as follows:
\begin{theorem}\label{thm:3.8}
	Suppose Assumption \ref{p:3.1} holds. Consider the iterative form \eqref{IM}, where ${{\alpha }_{k}}$ is determined by the modified Wolfe conditions and $d_k$ is given by \eqref{eq:dk+1} containing $\beta _{k}$ \eqref{eq:beta k+1} and $\theta_{k+1}$ \eqref{eq:thetak+1}, then
	\begin{equation}
		\underset{k\to \infty }{\mathop{\lim }}\,\inf ||{{g}_{k}}||=0,
	\end{equation}
	i.e. SCGMMWLS is globally convergent.
\end{theorem}
\begin{proof}
	Using Property (*) and the conclusion yielded by \eqref{eq:59} that $||d_k||^2$ grows at most linearly, we can show similarly to Lemma 4.2 in \cite{ref20} that there exists $\lambda>0$ such that, for any $\Delta \in N^{*}$ and any index $k_{0}$, there is a greater index $k \ge k_{0}$ such that $|\mathcal{K}_{k,\Delta}^{\lambda}|>\frac{\Delta}{2}$, where $\mathcal{K}_{k,\Delta}^{\lambda}:=\{i\in N^{*}: k \le i \le k+\Delta -1, ||s_{i-1}||>\lambda \}$ and $|\mathcal{K}_{k,\Delta}^{\lambda}|$ denote the number of elements of $\mathcal{K}_{k,\Delta}^{\lambda}$. By the above, Lemma \ref{lem:3.6} and the boundedness of $\{x_k\}$ \eqref{eq:51}, we can obtain a contradiction similarly to the proof of Theorem 4.3 in \cite{ref20} that $2M \ge \frac{1}{2} \sum_{i=k}^{k+\Delta-1} ||s_{i-1}||>\frac{\Delta}{2}|\mathcal{K}_{k,\Delta}^{\lambda}|>\frac{\lambda \Delta}{4} \ge 2M$, where $\Delta$ is chosen as $\lceil \frac{8M}{\lambda} \rceil$. This contradiction supports the truth of \eqref{eq:58}.
\end{proof}

\section{Numerical experiments}

In this section, we report some numerical comparison results of Algorithm SCGMMWLS, denoted M1, versus the following three other algorithms. \\
\\
DK: The algorithm proposed by Dai and Kou;\\
Jian: The algorithm proposed by Jian;\\
M1: Our algorithm SCGMMWLS; \\
M2: The spectral CG algorithm with standard Wolfe line search and $d_{k+1}$ is given by
\begin{align}
& {{d}_{k+1}}=-{{\theta }_{k+1}}{{g}_{k+1}}+{{\beta }_{k+1}}{{d}_{k}},\text{ }{{d}_{0}}=-{{g}_{0}}, \notag\\
& {{\beta }_{k+1}}=\max \{\frac{g_{k+1}^{T}{v_{k}^{(m)}}}{d_{k}^{T}{v_{k}^{(m)}}}-\frac{||{v_{k}^{(m)}}|{{|}^{2}}}{d_{k}^{T}{v_{k}^{(m)}}}\frac{g_{k+1}^{T}{{d}_{k}}}{d_{k}^{T}{v_{k}^{(m)}}}, \frac{g_{k}^{T}{{d}_{k}}}{||{{d}_{k}}|{{|}^{2}}}\}, \notag \\
& {{\theta }_{k+1}}=
\begin{cases}
{{{\tilde{\theta }}}_{k+1}}, & if \text{ } {{{\tilde{\theta }}}_{k+1}}\in [\frac{1}{4} +\eta ,\tau ]\text{,}\notag  \\
1, & otherwise\text{,} \notag\\
\end{cases} \notag
 \notag \\
& {{{\tilde{\theta }}}_{k+1}}=\frac{1}{{{g}_{k+1}^{T}}v_{k}^{(m)}}(s_{k}^{T}{{g}_{k+1}}+{{\beta }_{k+1}}{{d}_{k}^{T}}v_{k}^{(m)}), \notag \\
&v_{k}^{(m)}=y_{k}+\frac{m}{m-2}\frac{max\{\mu_{k},0\}}{\Vert s_{k} \Vert ^{2}}s_{k},\text{ } m\geq 3, \text{ } m\in Z,\notag \\
&\mu_{k}=2(f_{k}-f_{k+1})+(g_{k}+g_{k+1})^{T}s_{k},\notag
\end{align}

All experiments are coded in VC++6.0 and run on a laptop with Inter Core i5-9300H CPU, 16GB RAM memory and the Windows 10 operating system. Our test problems are extracted from a collection of unconstrained optimization test functions \cite{ref13} with variable dimensions from 100 to 10000.

The related parameters are set as follows:
\begin{equation}
\eta =0.001,\tau =10,(\rho,\sigma)=
\begin{cases}
(0.18,0.2),\text{ }if \text{ } M1, \notag \\
(0.1,0.9),\text{ }otherwise. \notag \\
\end{cases} .	
\end{equation}
The stopping criterion is that $||{{g}_{k}}|{{|}_{\infty }}\le {{10}^{-8}}$ or the number of iterations exceeds 10000.

It is known that the number of iterations (NI), the number of function evaluations (NF), the number of gradient evaluations (NG) and so on are important factors to judge the numerical performance of CG-type algorithms. We use the performance profiles of Dolan and More \cite{ref22} to evaluate the performance of algorithms. Define $\mathcal{P}$ and $\mathcal{S}$ respectively as the set consisting of $n_p$ test problems and the set of compared solvers. Define $NI_{p,s},NF_{p,s}$ and $NG_{p,s}$ to respectively be the number of iterations, the number of function evaluations and the number of gradient evaluations for solver $s$ solving problem $p$. Define the performance ratio as
\begin{equation}
	r_{p,s}^{I}=\frac{NI_{p,s}}{NI_{p}^{*}},\text{ }r_{p,s}^{F}=\frac{NF_{p,s}}{NF_{p}^{*}},\text{ }r_{p,s}^{G}=\frac{NG_{p,s}}{NG_{p}^{*}}, \notag
\end{equation}
where
\begin{equation}
	NI_{p}^{*}=min\{NI_{p,s},s \in \mathcal{S} \},\text{ }NF_{p}^{*}=min\{NF_{p,s},s \in \mathcal{S} \},\text{ }NG_{p}^{*}=min\{NG_{p,s},s \in \mathcal{S} \}. \notag
\end{equation}
It is obvious that $r_{p,s}^{I},r_{p,s}^{F},r_{p,s}^{G} \ge 1$ for all $p,s$. If a solver fails to solve a problem, then the ratio $r_{p,s}^{I},r_{p,s}^{F},r_{p,s}^{G}$ will be assigned to a lager number. We define the following cumulative distribution function, related to the performance ratio, which displays the performance of each solver $s$. Take for example $r_{p,s}^{I}$.
\begin{equation}
	\rho_{s}^{I}(\tau)=\frac{size\{ p \in \mathcal{P},r_{p,s}^{I} \le \tau \}}{n_p}. \notag
\end{equation}
$\rho_{s}^{I}(1)$ implies that the percentage of problems where solver $s$ wins.

Figure \ref{order} presents the performance profiles of SCGMMWLS with $m=\infty,5,4,3$. We observe that SCGMMWLS($m=3$) is obviously superior to the others, winning about 85\% of test problems and standing in the first place; Although the number of test problems SCGMMWLS($m=4$) wins is more than that of the other two, SCGMMWLS($m=5$) and SCGMMWLS($m=\infty$), it performs overall less well than they do. It is also worth noting that SCGMMWLS($m=5$) and SCGMMWLS($m=\infty$) have nearly identical performance. These results indicate that SCGMMWLS with $m=3$ is actually most efficient, but the performance does not get worse with growth of order $m$.

\begin{figure}[H]
\centering
\begin{minipage}[t]{0.3\textwidth}
\centering
\includegraphics[width=4.6cm]{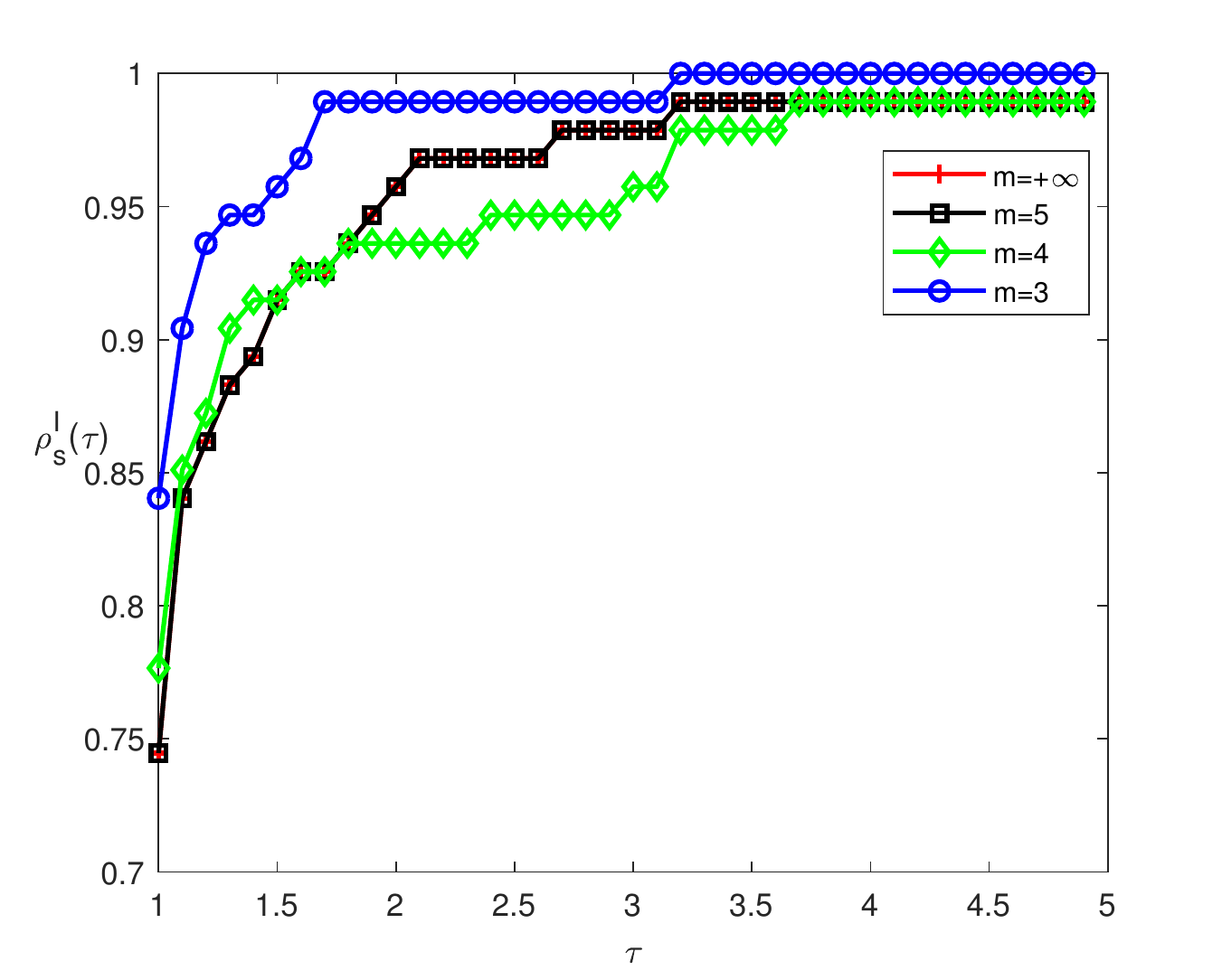}
\end{minipage}
\begin{minipage}[t]{0.3\textwidth}
\centering
\includegraphics[width=4.6cm]{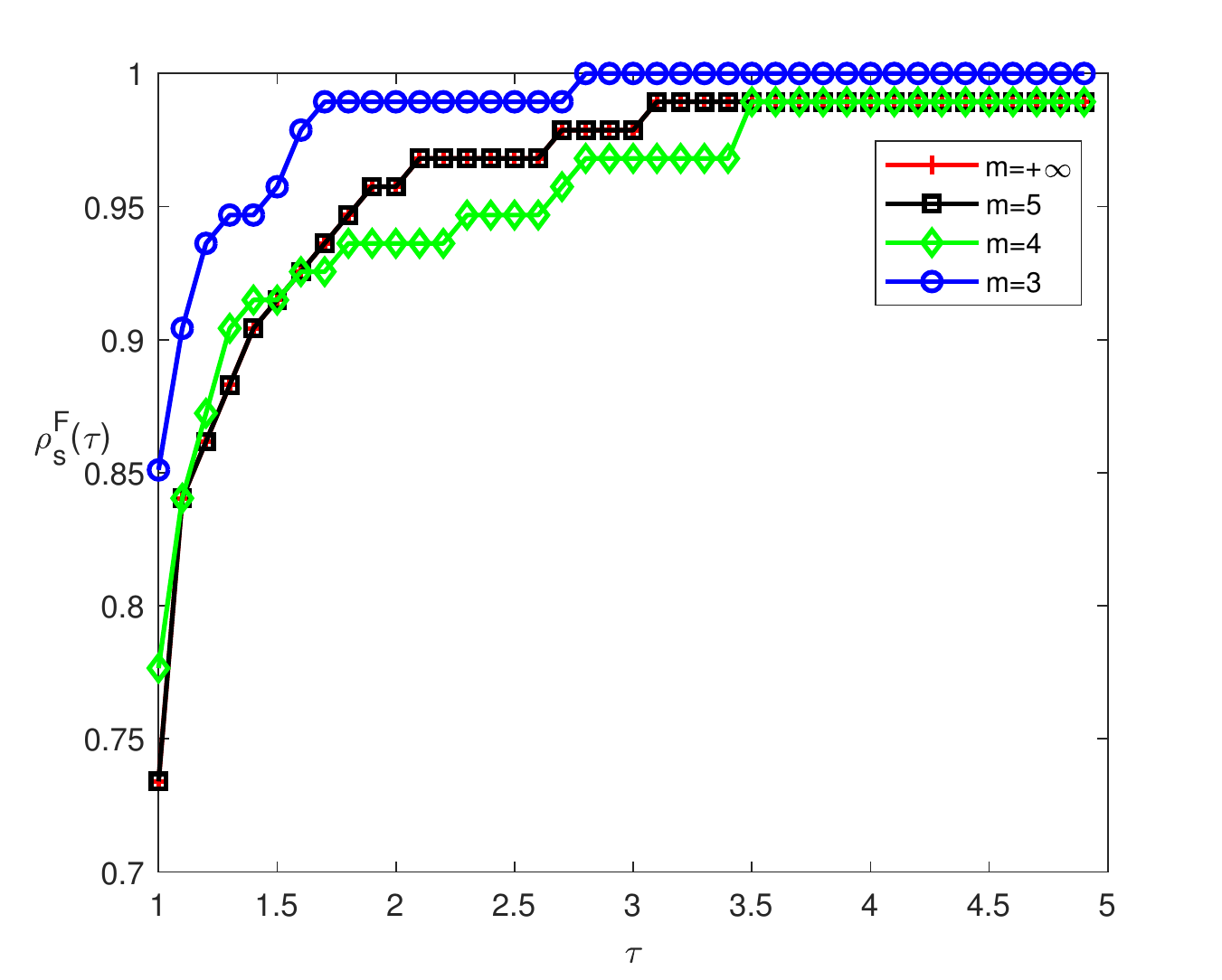}
\end{minipage}
\begin{minipage}[t]{0.3\textwidth}
\centering
\includegraphics[width=4.6cm]{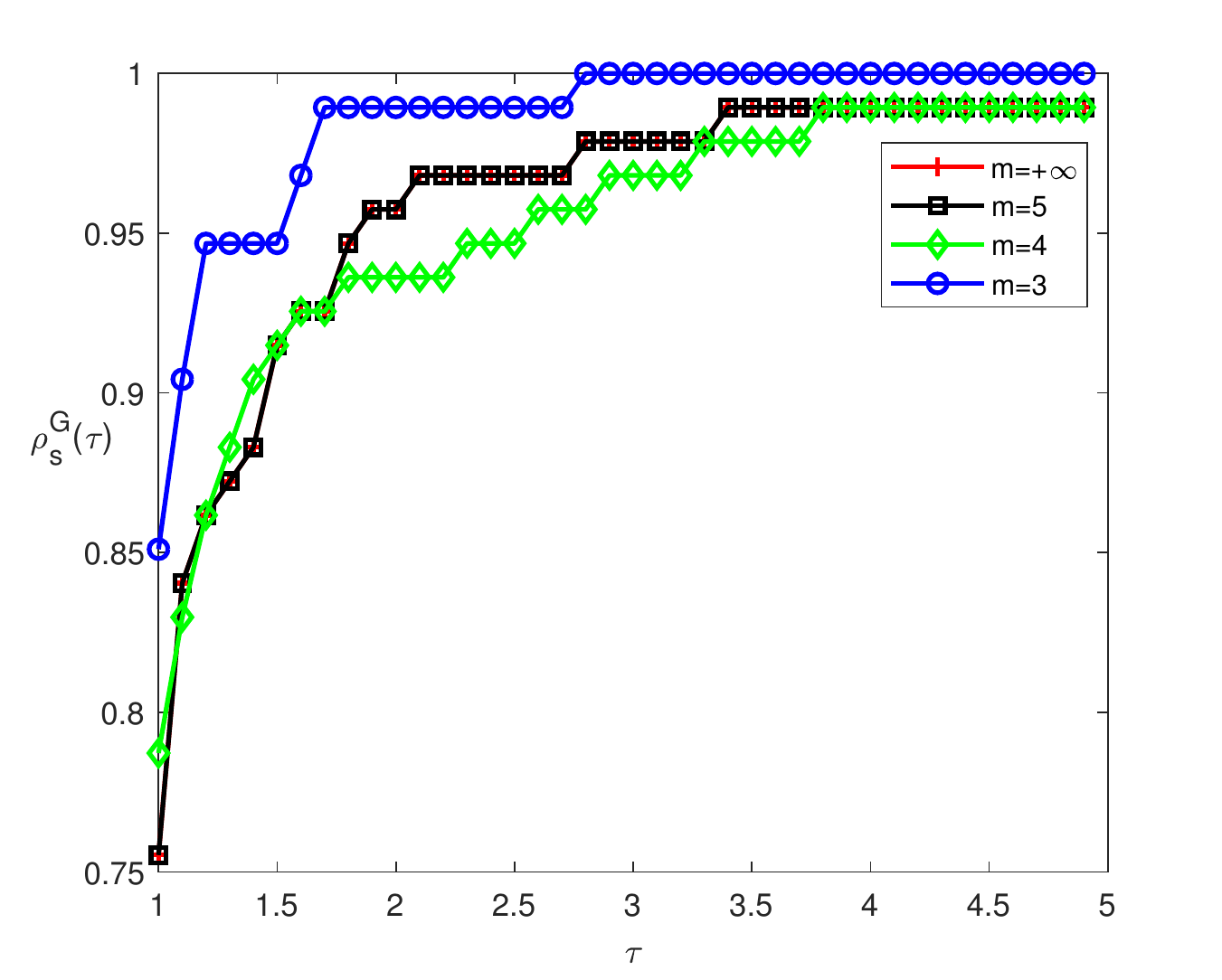}
\end{minipage}
\caption{Performance profiles of M1 with $m=\infty,5,4,3$  based on the number of iterations, function evaluations and gradient evaluations}
\label{order}
\end{figure}

M2 is actually SCGMMWLS with $t_k$ in \eqref{eq:Bkskl} replaced by $\frac{m}{m-2}\frac{max\{\mu_{k},0\}}{\Vert s_{k} \Vert ^{2}}$ and the Modified Wolfe line search replaced by the standard Wolfe line search. To keep the curvature condition holding, M2 applies the scheme that $max\{\mu_k,0\}$ supersedes $\mu_k$, abandoning negative $\mu_k$. Based on the results of the first experiment, we compare M1 with M2 in case $m=3$. In Figure \ref{M1M2}, M1 respectively wins about 80\%, 70\% and 63\% of test problems in NI, NF and NG while M2 wins about 35\%, 45\% and 52\% of test problems. M1 has a appreciable superiority, which means using negative $\mu_k$ can make further improvement on the performance of the algorithm. It is natural that using negative $\mu_k$ makes the modified Wolfe line search slightly stricter, may be causing more function and gradient evaluations each iteration. However, it can contribute to a more appropriate next iterate under the current iterate and less total function and gradient evaluations.

\begin{figure}[H]
\centering
\begin{minipage}[t]{0.3\textwidth}
\centering
\includegraphics[width=4.6cm]{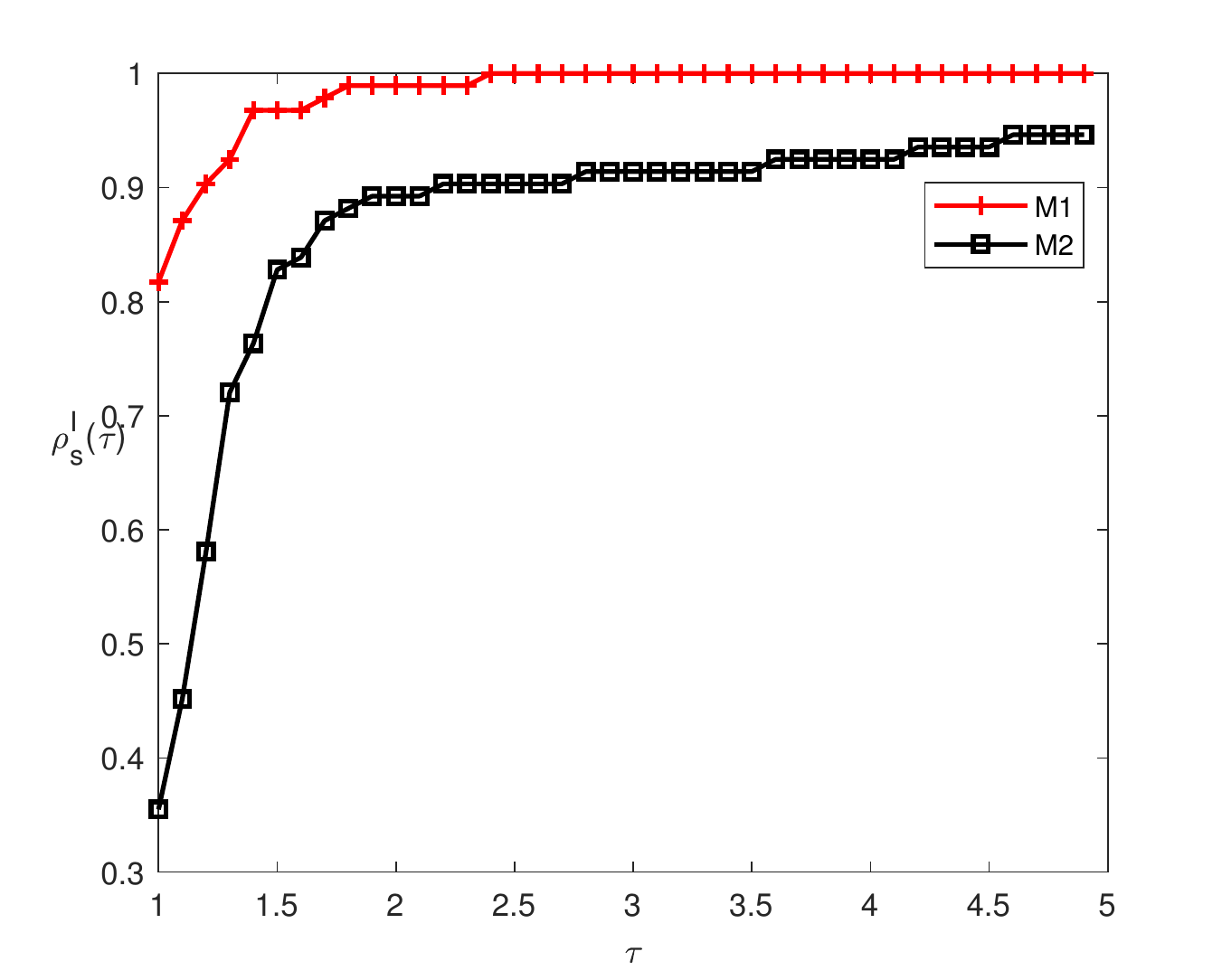}
\end{minipage}
\begin{minipage}[t]{0.3\textwidth}
\centering
\includegraphics[width=4.6cm]{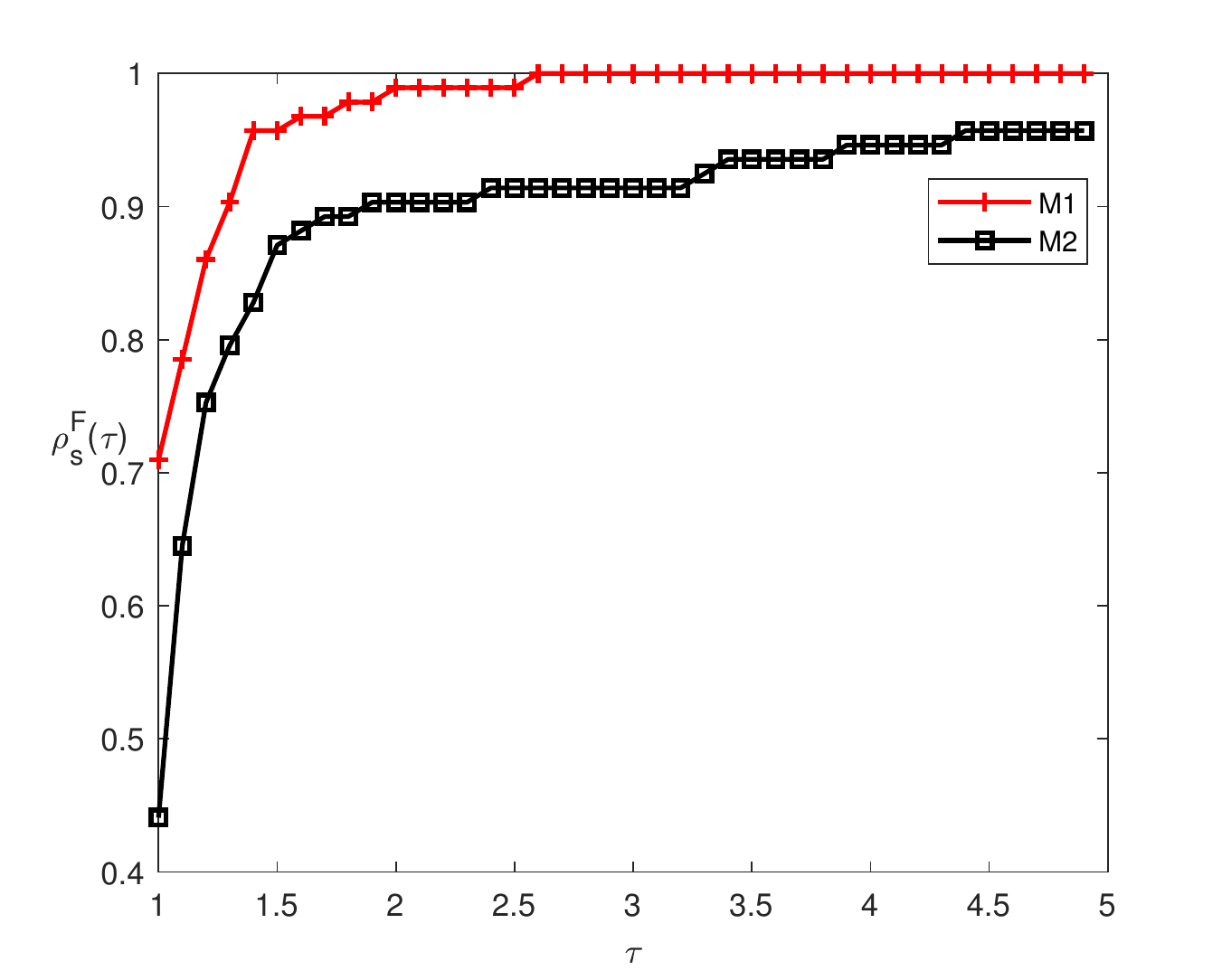}
\end{minipage}
\begin{minipage}[t]{0.3\textwidth}
\centering
\includegraphics[width=4.6cm]{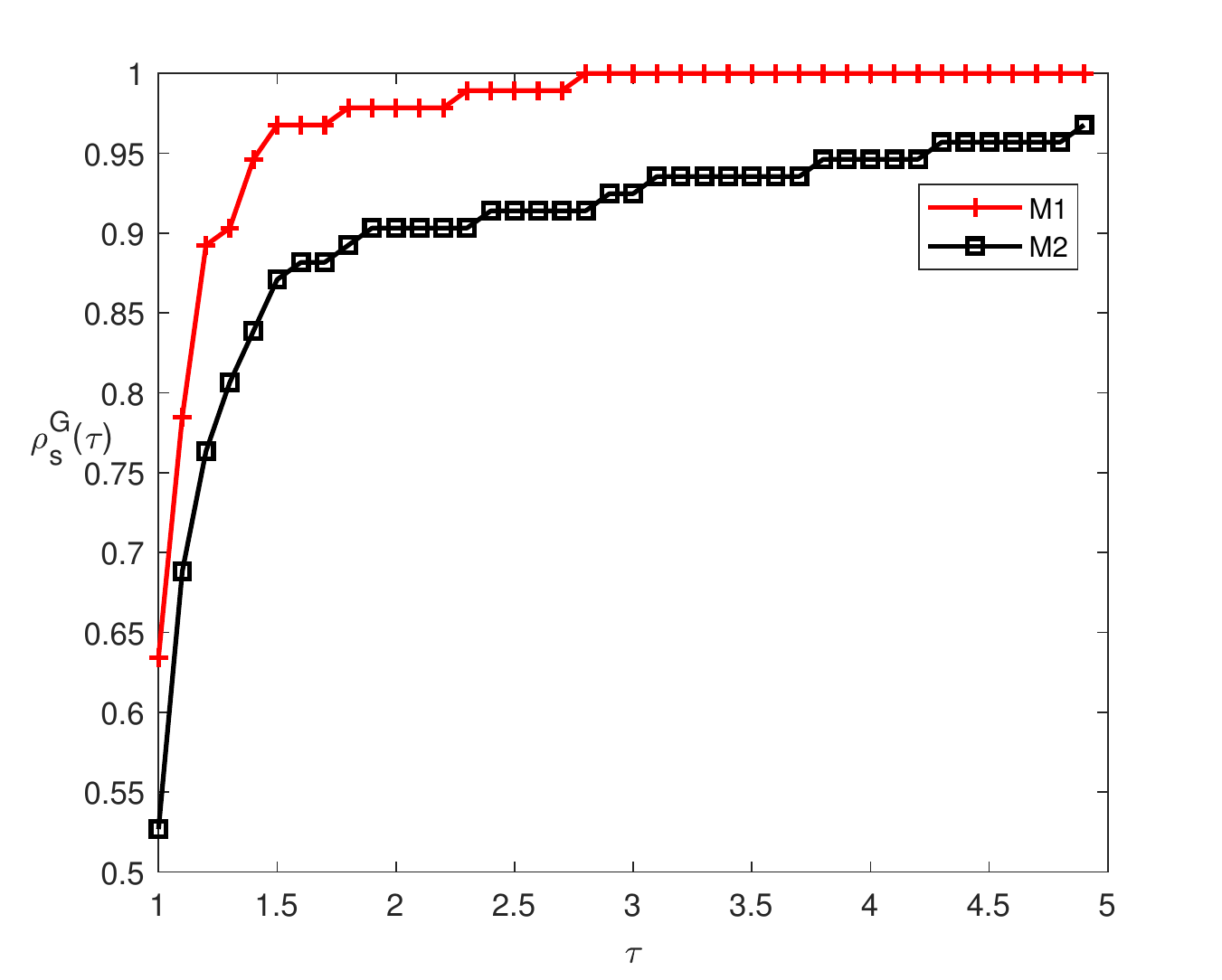}
\end{minipage}
\caption{Performance profiles of M1 and M2  based on the number of iterations, function evaluations and gradient evaluations}
\label{M1M2}
\end{figure}

We also compare M1($m=3$) with DK and Jian. From Figure \ref{DJO}, we can see that M1 is fastest for about 75\%, 60\% and 55\% of test problems in NI, NF and NG, occupying the first place. M1 obviously performs better than DK and Jian do.

\begin{figure}[H]
\centering
\begin{minipage}[t]{0.3\textwidth}
\centering
\includegraphics[width=4.6cm]{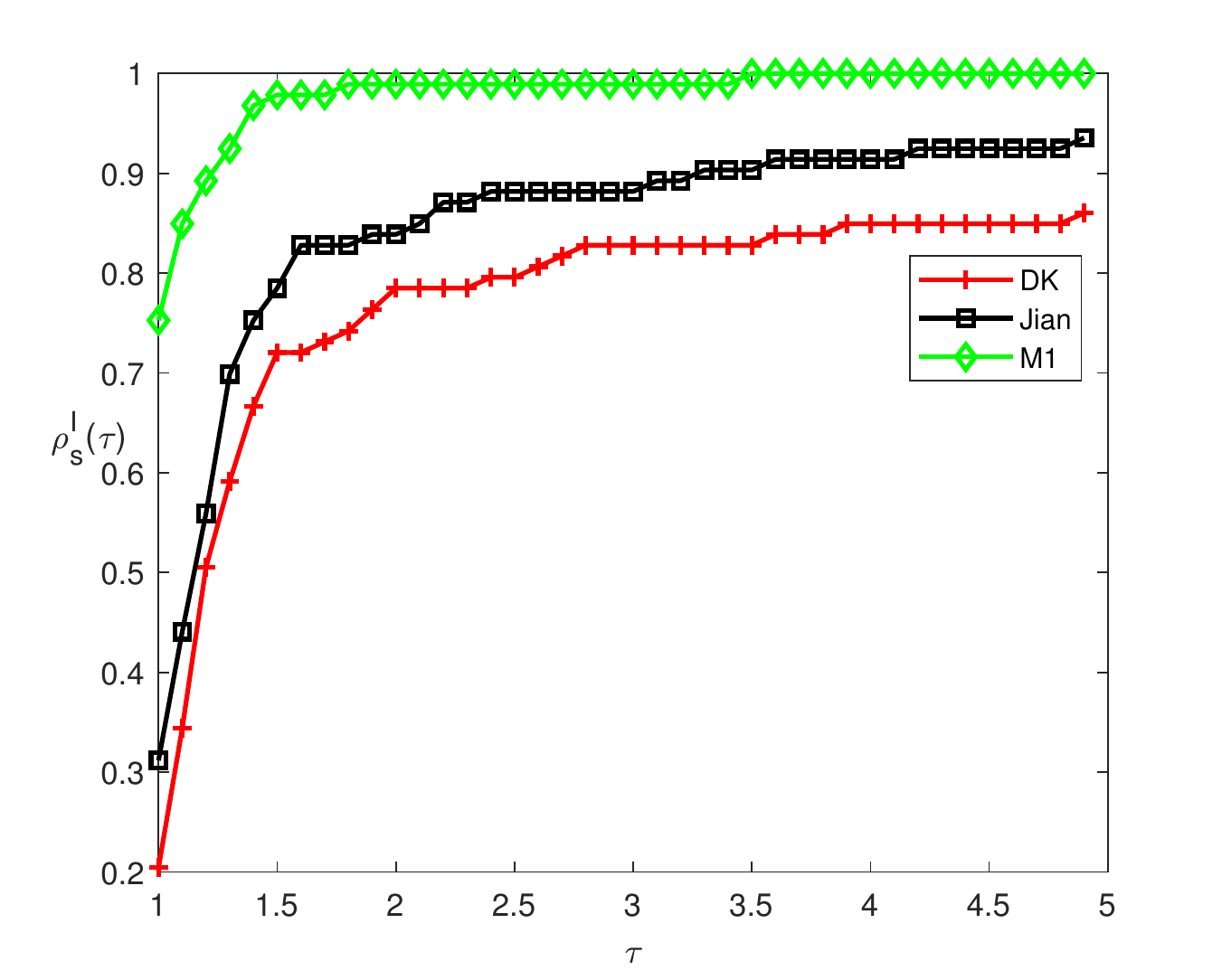}
\end{minipage}
\begin{minipage}[t]{0.3\textwidth}
\centering
\includegraphics[width=4.6cm]{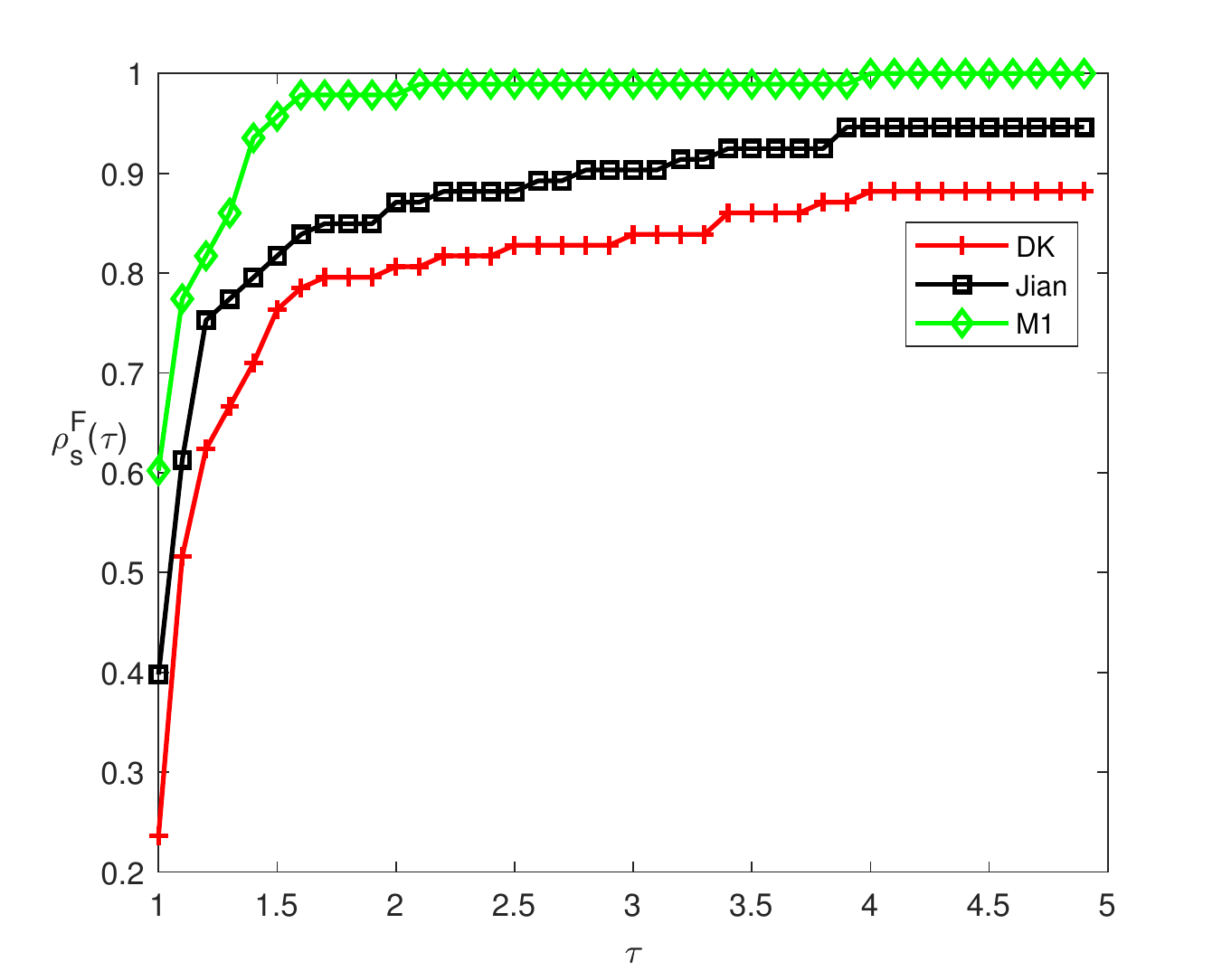}
\end{minipage}
\begin{minipage}[t]{0.3\textwidth}
\centering
\includegraphics[width=4.6cm]{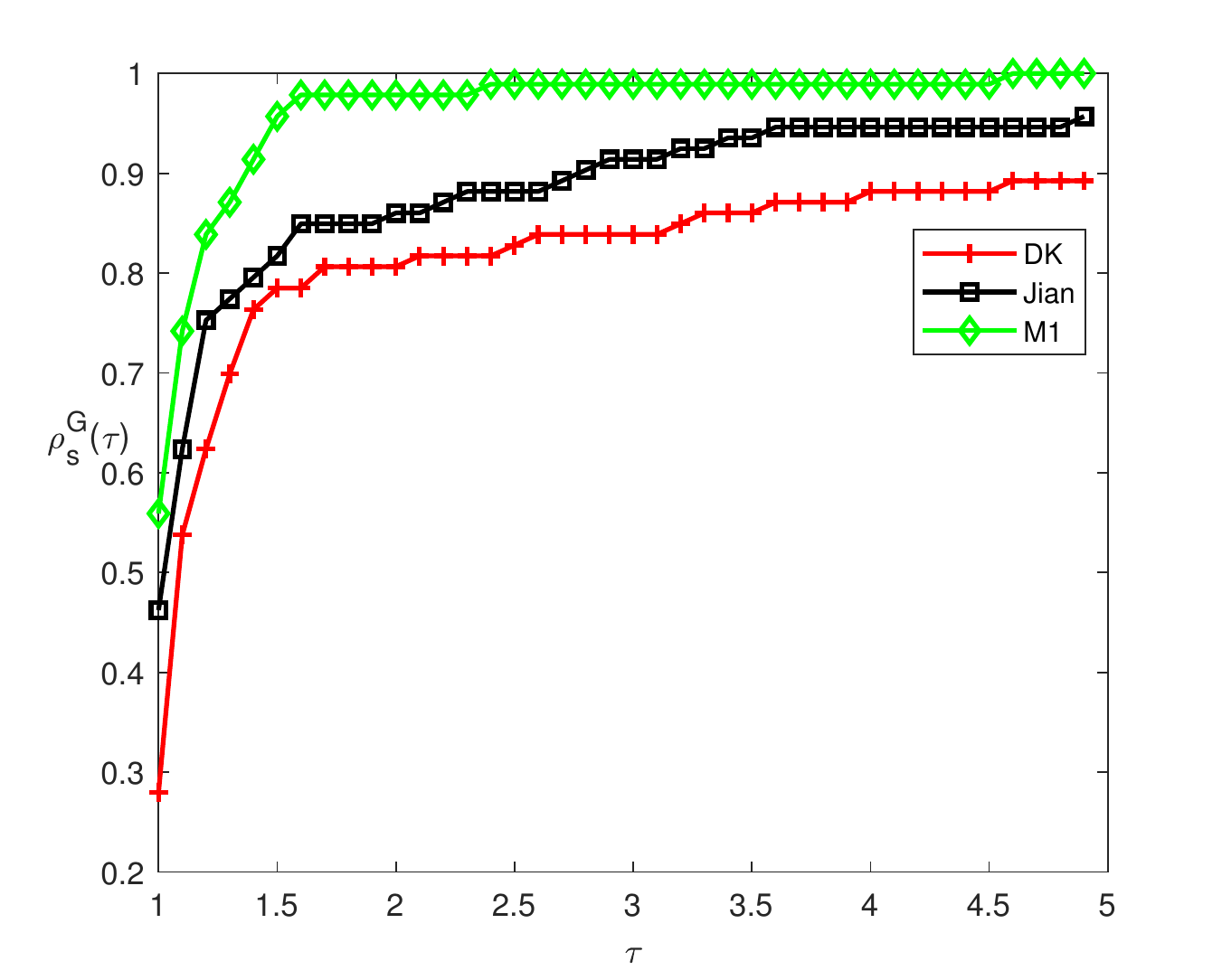}
\end{minipage}
\caption{Performance profiles of DK, Jian and M1 based on the number of iterations, function evaluations and gradient evaluations}
\label{DJO}
\end{figure}

\iffalse
\begin{figure}
	\centering
	\includegraphics[width=5in, keepaspectratio]{M2.eps}
	\caption{Performance profiles of M1 and M2  based on the number of iterations, function evaluations and gradient evaluations}
	\label{M1M2}
\end{figure}
\fi
\iffalse
\begin{figure}
	\centering
	\includegraphics[width=5in, keepaspectratio]{DK.eps}
	\caption{Performance profiles of DK, Jian and M1 based on the number of iterations, function evaluations and gradient evaluations}
	\label{DJO}
\end{figure}
\fi
\newpage

\section{Conclusions}
In this paper, based on a new family of modified secant equations and a modified Wolfe line search, a novel spectral CG algorithm has been proposed for unconstrained optimization problems.

According to the mth-order Taylor expansion of the objective function and cubic Hermite interpolation conditions, we derive a family of modified secant equations with higher accuracy in approximation of the Hessian matrix of the objective function. It contains previous a series of modified variants. To keep the modified secant equation always meeting the curvature condition without forcing the negative $\mu_k$ to be assigned 0, we develop a modified Wolfe line search. Based on this line search and Jian's spectral CG algorithm, a new improved spectral CG algorithm, SCGMMWLS, is presented and the global convergence of SCGMMWLS is established for general nonlinear functions.

By using the performance profile introduced by Dolan and More, on a collection of unconstrained optimization test functions, we verify that SCGMMWLS performs best in case $m=3$; We compare SCGMMWLS to the algorithm with $max\{\mu_k,0\}$ and standard Wolfe line search and find SCGMMWLS has a superior performance. It means that using negative $\mu_k$ can make further improvement on the performance of the algorithm. Besides, comparative experiments among SCGMMWLS, the algorithm proposed by Jian and the algorithm from Dai and Kou are conducted. The numerical results indicate that SCGMMWLS outperforms the other two algorithms.


\clearpage
\begin{thebibliography}{99}% Reference smaple:  http://www.ybook.co.jp/pjo-aim.pdf
\bibitem{ref1}Fletcher R , Reeves C M . Function minimization by conjugate gradients[J]. Computer Journal 1964, 7(6): 163-168.
\bibitem{ref2}Polyak B T . The conjugate gradient method in extremal problems[J]. Ussr Computational Mathematics and Mathematical Physics, 1969, 9(4):94-112
\bibitem{ref3}Hestenes M R, Stiefel E L. Methods of conjugate gradients for solving linear systems[J]. Journal of Research of the National Bureau of Standards (United States), 1952, 49: 409-436.
\bibitem{ref4}Dai Y H, Yuan Y X, A nonlinear conjugate gradient with a strong global convergence properties[J]. SIAM Journal on Optimization, 1999, 10(1): 177-182.
\bibitem{ref5}Hu Y F, Storey C. Efficient generalized conjugate gradient algorithms, part 2: Implementation[J]. Journal of Optimization Theory and Applications, 1991, 69(1):139-152.
\bibitem{ref6}Perry A. A Modified Conjugate Gradient Algorithm[J]. Operations Research, 1978, 26(6):1073-1078.
\bibitem{ref7}Dai Y H , Kou C X . A nonlinear conjugate gradient algorithm with an optimal property and an improved Wolfe line search[J]. SIAM Journal on Optimization, 2013, 23(1):296-320.
\bibitem{refp}Perry A. A class of conjugate gradient algorithms with a two-step variable metric memory[R]. Discussion paper, 1977.
\bibitem{refs}Shanno D F. On the convergence of a new conjugate gradient algorithm[J]. SIAM Journal on Numerical Analysis, 1978, 15(6): 1247-1257.
\bibitem{refzhang}Hager W W, Zhang H. A survey of nonlinear conjugate gradient methods[J]. Pacific Journal of Optimization, 2006, 2(1): 35-58.
\bibitem{ref8}Barzilai J, Borwein J M. Two-Point step size gradient methods[J]. IMA Jouenal of Numerical Analysis, 1988, 8(1):141-148.
\bibitem{ref9}Raydan M. The Barzilai and Borwein gradient method for the large scale unconstrained minimization problem[J]. SIAM Journal on Optimization, 1997, 7(1):26-33.
\bibitem{ref10}Birgin E G, J. M. Mart¨ªnez. A spectral conjugate gradient method for unconstrained optimization[J]. Applied Mathematics and Optimization, 2001, 43(2):117-128.
\bibitem{ref11}Andrei N. New accelerated conjugate gradient algorithms as a modification of Dai-Yuan's computational scheme for unconstrained optimization[J]. Journal of Computational and Applied Mathematics, 2010, 234(12):3397-3410.
\bibitem{ref12}Jian J, Chen Q, Jiang X, et al. A new spectral conjugate gradient method for large-scale unconstrained optimization[J]. Optimization Methods and Software, 2017, 32(3): 503-515.
%\bibitem{ref13}Li D H, Fukushima M. A modified BFGS method and its global convergence in nonconvex minimization[J]. Journal of Computational and Applied Mathematics, 2001, 129(1-2): 15-35.
%\bibitem{ref14}Parvaneh F, Keyvan A. A Modified Conjugate Gradient Method with Global Convergence[J]. Journal of Optimization Theory and Applications, 2019, 182: 667-690.
\bibitem{ref13}Andrei N. An unconstrained optimization test functions collection[J]. Adv. Model. Optim, 2008, 10(1): 147-161.
\bibitem{ref15}Yuan Y. A modified BFGS algorithm for unconstrained optimization[J]. IMA Journal of Numerical Analysis, 1991, 11(3): 325-332.
\bibitem{ref16}Khoshgam Z, Ashrafi A. A new modified scaled conjugate gradient method for large-scale unconstrained optimization with non-convex objective function[J]. Optimization Methods and Software, 2019, 34(4): 783-796.
\bibitem{ref17}Biglari F, Hassan M A, Leong W J. New quasi-Newton methods via higher order tensor models[J]. Journal of Computational and Applied Mathematics, 2011, 235(8): 2412-2422.
\bibitem{ref18}Yuan Y, Byrd R H. Non-quasi-Newton updates for unconstrained optimization[J]. Journal of Computational Mathematics, 1995: 95-107.
\bibitem{refB}Babaie-Kafaki S, Ghanbari R, Mahdavi-Amiri N. Two new conjugate gradient methods based on modified secant equations[J]. Journal of Computational and Applied Mathematics, 2010, 234(5): 1374-1386.
\bibitem{refZ}Sun Z, Li H, Wang J, et al. Two modified spectral conjugate gradient methods and their global convergence for unconstrained optimization[J]. International Journal of Computer Mathematics, 2018, 95(10): 2082-2099.
\bibitem{ref19}Zoutendijk G. Nonlinear programming, computational methods[J]. Integer and Nonlinear Programming, 1970: 37-86.
\bibitem{ref20}Gilbert J C, Nocedal J. Global convergence properties of conjugate gradient methods for optimization[J]. SIAM Journal on Optimization, 1992, 2(1): 21-42.
\bibitem{ref21}Kou C X. An improved nonlinear conjugate gradient method with an optimal property[J]. Science China Mathematics, 2014, 57(3): 635-648.
\bibitem{ref22}Dolan E D, More J J. Benchmarking optimization software with performance profiles[J]. Mathematical Programming, 2002, 91(2): 201-213.
\end{thebibliography}
\end{document}